\documentclass[10pt]{article}
\usepackage[left=2cm,right=2cm,top=3cm]{geometry}

\providecommand{\keywords}[1]{\textbf{\textit{Keywords--}} #1}

\providecommand{\subjclass}[1]{\textbf{\textit{MSC 2020---}} #1}

\usepackage[english]{babel}
\usepackage[utf8]{inputenc}
\usepackage{amsmath,amssymb,amsthm}
\usepackage{thmtools}
\usepackage[all,cmtip]{xy}
\usepackage[bookmarks=true]{hyperref}
\usepackage{tikz}
\usepackage{tikz-cd}
\usepackage{todonotes}
\usepackage{mathtools}
\usepackage{dsfont}
\usepackage{tensor}
\usepackage{verbatim}
\usepackage{graphicx}
\usepackage[shortlabels]{enumitem}
\usepackage{xstring}
\usepackage[capitalize]{cleveref}
\usepackage{tensor}
\usepackage{mathrsfs}
\usepackage[mathscr]{euscript}
\usepackage{nicefrac}

\definecolor{darkolivegreen}{rgb}{0.33, 0.42, 0.18} 
\definecolor{cobalt}{rgb}{0.0, 0.24, 0.43}

\newcounter{comments}
\newenvironment{displaycomment}{\begin{list}{}{\rightmargin=1cm\leftmargin=1cm}\item\sf\begin{small}}{\end{small}\end{list}}

\newcommand{\C}{\mathbb{C}}
\newcommand{\R}{\mathbb{R}}
\newcommand{\Z}{\mathbb{Z}}

\newcommand{\mc}[1]{\mathcal{#1}}

\newcommand{\cstar}{C$^{*}$}

\newcommand{\Pol}{\mathrm{Pol}}

\newcommand{\pullb}{\star}

\newcommand*{\defeq}{\mathrel{\vcenter{\baselineskip0.5ex \lineskiplimit0pt
\hbox{\scriptsize.}\hbox{\scriptsize.}}}%
=}

\newcommand{\ph}{\varphi}

\newcommand{\GL}{\operatorname{GL}}

\newcommand{\U}{\operatorname{U}}

\newcommand{\Cl}{\operatorname{Cl}}

\newcommand{\gra}{\operatorname{graph}}

\renewcommand{\1}{\mathds{1}}

\newcommand{\Sp}{\operatorname{Sp}}

\newcommand{\bose}{H^{\mathrm{b}}}
\newcommand{\fermi}{H^{\mathrm{f}}}
\newcommand{\HS}{\mc{B}_{2}} 

\newcommand\blfootnote[1]{%
  \begingroup
  \renewcommand\thefootnote{}\footnote{#1}%
  \addtocounter{footnote}{-1}%
  \endgroup
}

\newcommand{\sphere}{\overline{\mathbb{C}}}
\newcommand{\disk}{\mathbb{D}}
\renewcommand{\O}{\operatorname{O}}

\renewcommand{\hat}[1]{\widehat{#1}}

\makeatletter
\def\thm@space@setup{%
\thm@preskip=0.6em
\thm@postskip=0.00em
}
\makeatother

\declaretheoremstyle[
headfont=\normalfont\bfseries,
bodyfont=\normalfont,
qed={$\triangle$},
]{definitionstyle}

\declaretheorem[
style=definitionstyle,
title=Definition,
refname={Definition,Definitions},
Refname={Definition,Definitions},
numberwithin=section,
]{definition}
\declaretheorem[
style=definitionstyle,
title=Remark,
refname={Remark,Remarks},
Refname={Remark,Remarks},
sibling=definition,
]{remark}
\declaretheorem[
style=definitionstyle,
title=Example,
refname={Example,Example},
Refname={Example,Examples},
sibling=definition,
]{example}

\declaretheorem[
style=definitionstyle,
title=Exercise,
refname={Exercise,Exercises},
Refname={Exercise,Exercises},
sibling=definition,
]{exercise}

\theoremstyle{plain}
\newtheorem{theorem}[definition]{Theorem}
\newtheorem{proposition}[definition]{Proposition}
\newtheorem{lemma}[definition]{Lemma}
\newtheorem{corollary}[definition]{Corollary}

\setlist{topsep=-0.2em,itemsep=0em}

\crefname{enumi}{\unskip}{\unskip}

\title{Grassmannians of Lagrangian Polarizations}

\author{Peter Kristel\footnote{Hausdorff Center for Mathematical Sciences, Bonn; pkristel@gmail.com}\; \& Eric Schippers\footnote{University of Manitoba, Winnipeg; eric.schippers@umanitoba.ca}}

\date{\today}

\makeatletter
\hypersetup{
colorlinks=true,
urlcolor=cobalt,
linkcolor=cobalt,
citecolor=darkolivegreen,
pdfmenubar=false,
pdftoolbar=true,
pdftitle = {\@title},
}
\makeatother

\renewcommand{\thefootnote}{\roman{footnote}}

\begin{document}

\maketitle

\tableofcontents

\abstract{This paper is an introduction to polarizations in the symplectic and orthogonal settings. They arise in association to a triple of compatible structures on a real vector space, consisting of an inner product, a symplectic form, and a complex structure. A polarization is a decomposition of the complexified vector space into the eigenspaces of the complex structure; this information is equivalent to the specification of a compatible triple.  When either a symplectic form or inner product is fixed, one obtains a Grassmannian of polarizations.
We give an exposition of this circle of ideas, emphasizing the symmetry of the symplectic and orthogonal settings, and allowing the possibility that the underlying vector spaces are infinite-dimensional. This introduction would be useful for those interested in applications of polarizations to representation theory, loop groups, complex geometry, moduli spaces, quantization, and conformal field theory.}\blfootnote{\keywords{Lagrangian Grassmannian, polarizations, symplectic form, inner product, Siegel disk, symplectic group, orthogonal group}}\blfootnote{\subjclass{15A66, 15B77, 32G15, 53D99, 81S10}}

\section{Introduction}

This paper is a self-contained introduction to polarizations of complex vector spaces and Lagrangian Grassmannians of polarizations. These appear in geometry and algebra in many contexts, such as algebraic and complex geometry \cite{farkas_riemann_1992,siegel_topics_1988,birkenhake_complex_2004}, moduli spaces \cite{nag_teichmuller_1995,siegel_topics_1988}, loop groups \cite{segal_unitary_1981,pressley_loop_2003}, and the metaplectic and spin representations \cite{plymen_spinors_1994,habermann_introduction_2006}. They also appear in physics in association with the latter objects \cite{woit_quantum_2017,ottesen_infinite_1995}, and in conformal field theory \cite{huang_two-dimensional_1995,bleuler_definition_1988}. The paper covers the two types of polarization: the Riemannian one, in which the polarization is an orthogonal decomposition; and the symplectic one, in which the polarization is a symplectic decomposition.

Our presentation is suitable for graduate students. We also hope that seasoned researchers might find it useful as a quick introduction. It could serve as preparation or as a reference for those encountering polarizations and Grassmannians in any of the topics above. Although these topics are not treated here, we included many examples with references for orientation. 

Throughout, we have emphasized the symmetry of the Riemannian and symplectic point of view. We therefore chose a ``middle ground'' notationally, sampling from common notation in both fields, as well as from complex geometry. 

At the heart of the idea of polarization is a triple of compatible structures: a complex structure, an inner product, and a symplectic form. An familiar example is that of a K\"ahler manifold, which is a manifold equipped with three compatible structures: 1) a Riemannian metric; 2) an integrable complex structure; 3) a symplectic form. Here the three structures vary from point to point on the manifold. 

The compatibility condition ensures that when we are provided with two out of the three structures, we can reconstruct the third one, if it exists (which is not always the case). 
It then becomes natural to fix \emph{one} of these structures, and study the space of structures of a second type, which allow for the reconstruction of a compatible structure of the third type.
For example, we can fix a symplectic manifold $(M, \omega)$, and study the space of integrable complex structures on $M$, such that $M$ admits a compatible Riemannian metric. 

In this article we restrict our attention to structures on a fixed real vector space $M=H$. (In the example of K\"ahler manifold, one could restrict to a particular tangent space, or, one could consider the case that the manifold is just a complex vector space). In particular we do not discuss integrability.  Instead, we consider deformations of these structures on the fixed vector space, or equivalently, the Grassmannian of polarizations on the complexification of $H$.  We explore different models of the Grassmannians, and their manifold structures.
We also simultaneously generalize the setting, by allowing (and indeed focusing on) the case that $H$ is infinite-dimensional. Doing this requires only a little extra effort, and immediately broadens the scope. 
Some functional analytic issues rear their head when passing to the infinite-dimensional case, but these are easily dealt with.  
There are many interesting examples which fit in this context, many of which appear in the text. In general, most of the examples are completely accessible, but in a few cases where details would be distracting we indicate the relevant literature. 

Here is an outline of the paper.
In \cref{sec:Triples} we recall the definition of compatible triples, and explore their basic properties and characterizations. We also take care of a few functional analytic issues necessary when dealing with the case of (infinite-dimensional) Hilbert spaces.  

In \cref{sec:Polarizations}, we consider both types of polarization. Briefly, 
if $J$ is a complex structure on $H$, then the complexification $H_{\C}$ splits as a direct sum of the $\pm i$-eigenspaces of $J$.
This establishes a correspondence between certain types of decompositions of $H_{\C}$, and complex structures on $H$.
We exploit this relation to phrase our problem in the language of Grassmannians, in anticipation of Sections \ref{sec:Bosons} and \ref{sec:fermions}.
If $g$ is a Riemannian metric, and $J$ is orthogonal, then the corresponding decomposition of $H_{\C}$ is orthogonal.
If $\omega$ is a symplectic form, and $J$ is a symplectomorphism, then the corresponding decomposition of $H_{\C}$ is Lagrangian. 
 
A Lagrangian, resp.~orthogonal decomposition of $H_{\C}$ is part of the data required to carry out bosonic, resp.~fermionic geometric quantization.
In the current context, this procedure produces a representation of the Heisenberg, resp.~Clifford algebra of $H$ on the Fock space of the chosen subspace of $H_{\C}$.
We give a cursory overview of this construction in \cref{sec:HilbertSchmidt} for the sake of motivation.
For us, the most important point is that considerations of unitary equivalence of representations lead to a natural functional analytic condition, which picks out a subset of each of the Grassmannians under consideration. These are the ``restricted'' Grassmannians of representation theory. 

In \cref{sec:Bosons,sec:fermions} we describe the symplectic and orthogonal Grasmannians respectively, giving different models and providing a logically complete exposition of their basic properties. 
These Grassmannians have the desirable property that they can be equipped with the structure of complex manifold. 
We give a detailed construction of an atlas of these manifolds.  In the restricted case, this is modelled on the Hilbert space of Hilbert-Schmidt operators on some infinite-dimensional Hilbert space. 

Section \ref{se:sewing} explores a geometric interpretation of a particular Lagrangian Grassmannian in terms of sewing. This example arises in loop groups, conformal field theory, and Teichm\"uller theory.  Finally, Section \ref{se:solutions} contains a complete set of solutions to the exercises. 

\paragraph{Acknowledgements}

PK gratefully acknowledges support from the Pacific Institute for the Mathematical Sciences, and from the Hausdorff Center for Mathematics.  ES acknowledges the support of the Natural Sciences and Engineering Research Council of Canada (NSERC).

\section{Compatible triples}\label{sec:Triples}

In this section, we introduce the notion of compatible triples. This is the information of an inner product, symplectic form, and complex structure, which are related in a sense defined shortly. We will see various equivalent ways of expressing this compatibility. The ubiquity of compatible triples is illustrated with examples.  
As promised, we give an exposition in the generality of Hilbert spaces. 

We begin with a few considerations about Hilbert spaces.
Those concerned only in the finite-dimensional case could skip this first page, and ignore a few extra lines of justification in some of the proofs, without losing the thread. This could profitably be maintained until the restricted Grassmannians are considered starting in Section \ref{sec:Bosons}.  

Let $H$ be a separable (possibly even finite-dimensional) real Hilbert space.
That is, $H$ is a separable topological vector space over $\R$, which admits an inner product with respect to which it is a Hilbert space.
We recall the following three notions:
\begin{itemize}\itemsep-.2em
	\item A \emph{strong inner product} on $H$ is a continuous symmetric bilinear form on $H$, such that the map $\ph_{g}:H \rightarrow H^{*}$ defined by the relation $\ph_{g}(v)(w) = g(v,w)$ is an isomorphism, and such that $g(v,v) \geqslant 0$ for all $v \in H$.
	\item A \emph{complex structure} $J$ on $H$ is a continuous map $J: H \rightarrow H$ with the property that $J^{2} = -\1$. 
	\item A \emph{strong symplectic form} on $H$ is a continuous anti-symmetric bilinear form $\omega$ on $H$, such that the map $\ph_{\omega}:H \rightarrow H^{*}$ defined by the relation $\ph_{\omega}(v)(w) = \omega(v,w)$ is an isomorphism.
\end{itemize} 
\vspace{.4em}
The maps $\ph_{\omega},\ph_{g}: H \rightarrow H^{*}$ are sometimes called the musical isomorphisms.
Note that here, by isomorphism, we mean a bounded linear bijection. Thus the inverse must also be bounded. Of course in finite dimensions a linear bijection is automatically bounded.

    When a vector space is equipped with a symplectic form $\omega$, one may consider the \emph{canonical commutation relation} algebra of $H$.
    This is the unital, associative algebra generated by $H$ subject to the condition $vw - wv = \omega(v,w)$.
    Similarly, when one is given an inner product, one may consider the \emph{canonical anti-commutation relation} algebra of $H$.
    This is the unital, associative algebra generated by $H$ subject to the condition $vw + wv = g(v,w)$.
    Consideration of (certain variations of) these algebras and their modules motivates many of the definitions and results outlined in this article.
    We will give a sketch of some of this motivation in \cref{sec:HilbertSchmidt} ahead.

If $(H,g)$ is a Hilbert space, then $g$ is a strong inner product. In general, a strong inner product gives Hilbert space structure equivalent to $g$, as the following proposition shows.

\begin{proposition}\label{pr:StrongEquivalence}
    Let $(H,g)$ be a Hilbert space. Let $\psi: H \times H \rightarrow \R$ be a bilinear pairing. Then $\psi$ is a strong inner product if and only if $(H,\psi)$ is a Hilbert space and $\psi$ and $g$ are equivalent.
\end{proposition}
\begin{proof}
    Denote by $\| w \|_\psi$ the norm $\sqrt{ \psi(w,w)}$. By assumption there is a $K$ such that $\| w \|_\psi \leq K \| w\|$ for all $w$, and in particular $\| w \|_\psi$ is finite for any fixed $w$.  
    
    Since $\ph_\psi$ is an isomorphism, we have that it is bounded below.  
    Using this together with the Cauchy-Schwarz inequality for $\psi$ we get 
    \begin{equation*}
      C \| w \|  \leq \| \ph_{\psi}(w) \| = \sup_{\| v \|=1 } |\psi(w,v)| 
      \leq \| w \|_\psi \| v\|_\psi \leq K \| w\|_{\psi}.       
    \end{equation*}
    This shows that $(H,\psi)$ is complete and thus a Hilbert space, and that the norms are equivalent.  

    Conversely, if $(H,\psi)$ is a Hilbert space and the norms induced by $\psi$ and $g$ are equivalent, then they are also equivalent on $H^*$. Since  $\ph_\psi:(H,\psi) \rightarrow (H^*,\psi)$ is an isomorphism by the Riesz representation theorem, it is also a bounded isomorphism with respect to $g$.  
\end{proof}
Emboldened by \cref{pr:StrongEquivalence} we occasionally refer to a Hilbert space without designating a particular metric as special.

  It is worth pausing here to compare the infinite- and finite-dimensional settings.
    If $\psi: H \times H \rightarrow \R$ is a bilinear map, then $\psi$ is \emph{weakly non-degenerate} if the induced map $\ph_{\psi}: H \rightarrow H^{*}$ defined by $\ph_{\psi}(v)(w) = \psi(v,w)$ is injective.
    This is equivalent to the statement that for every $0 \neq v \in H$ there exists a $w \in H$ such that $\psi(v,w) \neq 0$.
    The map $\psi$ is \emph{strongly non-degenerate} if the induced map $\ph_{\psi}: H \rightarrow H^{*}$ is an isomorphism.
    In finite dimensions, these notions coincide, but if $H$ is infinite-dimensional, this is no longer true as \cref{ex:WeakInnerProduct} will show.
    On the other hand, in finite dimensions an inner product is automatically strong on account of its positive-definiteness.  
 
\begin{example}\label{ex:WeakInnerProduct}
    Consider the Hilbert space $\ell^{2}(\R)$ of square-summable sequences, with its standard inner product $g( \{ a_{n} \}, \{ b_{n} \}) = \sum_{n \geqslant 1} a_{n}b_{n}$.
    We equip $\ell^{2}(\R)$ with the symmetric bilinear map $\psi(\{a_{n}\},\{b_{n}\}) = \sum_{n \geqslant 1} a_{n}b_{n}/n$.
    The map $\psi$ is certainly weakly non-degenerate.
    It is however, not strongly non-degenerate, because $\ph_{\psi}$ is not surjective.
    Indeed, let $\xi: H \rightarrow \R$ be the map $\xi(\{a_{n} \}) = \sum_{n \geqslant 1} a_{n}/n$.
    The pre-image $\{ x_{n} \}$ of $\xi$ under the map $\ph_{\psi}$ (if it existed) would have to satisfy
    \begin{equation*}
        \sum_{n \geqslant 1} a_{n}/n = \psi(x_{n},a_{n}) = \sum_{n \geqslant 1} x_{n} a_{n} / n,
    \end{equation*}
    for all $\{ a_{n} \} \in \ell^{2}(\R)$.
    It follows that $x_{n} =1$ for all $n$, but this sequence is not square-summable.
\end{example}
	We will not be concerned with weakly non-degenerate maps, so from now on, we assume that any inner product or symplectic form is strong.
	The term will be dropped except when needed for proof or emphasis.

\begin{definition}\label{def:CompatTriples}
	A \emph{compatible triple} $(g,J,\omega)$ consists of a strong inner product $g$, a complex structure $J$, and a strong symplectic form $\omega$, such that $g(v,w) = \omega(v,Jw)$.
\end{definition}
Compatible triples can also be characterized as follows.
\begin{proposition}  \label{pr:three_properties_compatible} Let $(g,J,\omega)$ be an inner product, complex structure $J$, and symplectic form $\omega$ respectively. The following are equivalent.
	\begin{enumerate}
		\item $g(v,w) = \omega(v,Jw)$ for all $v,w \in H$;
		\item $\omega(v,w) = g(Jv,w)$ for all $v,w \in H$;
		\item $J(v) = \ph_{g}^{-1} \ph_{\omega}(v)$ for all $v \in H$.
	\end{enumerate} 
\end{proposition}
\begin{proof}
    Observe that properties 1 and 2 are equivalent to $\ph_{g} = -\ph_{\omega} J$ and $\ph_{\omega} = \ph_{g} J$ respectively.
    Now by multiplying with $J$ from the right, we see that $\ph_{g} = -\ph_{\omega} J$ if and only if $\ph_{\omega} = \ph_{g} J$, and by multiplying with $\ph_{g}^{-1}$ from the left, we see that $\ph_{\omega} = \ph_{g}J$ if and only if $\ph_{g}^{-1} \ph_{\omega} = J$.
\end{proof}
Thus, any of the above could be taken as equivalent definitions of compatible triple. 

\begin{exercise} \label{exer:one_two_then_J_preserves} Assume that  $(g,J,\omega)$ is a compatible triple.  Show that $J$ preserves both $g$ and $\omega$, i.e.~$g(Jv,Jw) = g(v,w)$ and $\omega(Jv,Jw) = \omega(v,w)$. 
\end{exercise}

\begin{exercise} \label{exer:two_of_three_gives_compat} Show that
	\begin{enumerate}[label=\alph*)]
	    \item If $g$ and $J$ are an inner product and a complex structure, then there exists a symplectic form $\omega$ such that $(g,J,\omega)$ is compatible if and only if $J$ is skew-adjoint with respect to to $g$.
	    \item If $g$ and $\omega$ are an inner product and a symplectic form, then there exists a complex structure $J$ such that $(g,J,\omega)$ is compatible if and only if $\ph_{g}^{-1}\ph_{\omega} = - \ph_{\omega}^{-1} \ph_{g}$.
	    \item If $J$ and $\omega$ are a complex structure and a symplectic form, then if $g(v,w) = \omega(v,Jw)$ defines an inner product, then $J$ is skew-symmetric with respect to $\omega$.
	\end{enumerate}
	\vspace{.4em}
 \end{exercise}
 \begin{exercise} \label{exer:example_not_inner}
	Show that there exist a symplectic form $\omega$ and a complex structure $J$, which is skew-symmetric with respect to $\omega$, such that $g(v,w) = \omega(v,Jw)$ does not define an inner product.
\end{exercise}

\begin{example} \label{ex:ct_standard_finite}
 Let $H=\mathbb{R}^{2n}$.
  Let $\{e_{i}\}_{i=1,...,2n}$ be the standard basis, and write $x_{i} = e_{2i-1}$ and $y_{i} = e_{2i}$.
 Let $g$ be the standard inner product, and define $J$ by $J x_{i} = y_{i}$ and $J y_{i} = - x_{i}$.
 The symplectic form $\omega$ such that $(g, J, \omega)$ is a compatible triple is given by 
 \[ \omega(x_k,y_k)= -\omega(y_k,x_k)=1, \ \ \ \ k=1,\ldots,n,   \]
 and $\omega$ is zero on all other pairs of basis vectors. 
\end{example}
\begin{example} \label{ex:complex_manifold}
 Let $M$ be a complex manifold; that is, a $2n$-dimensional manifold with an atlas of charts $\{ \phi:U \rightarrow \mathbb{C}^n \}$ such that for any pair of charts $\phi_1$, $\phi_2$, it holds that $\phi_2 \circ \phi_1^{-1}$ is a biholomorphism on its domain. If $(z_1,\ldots,z_n)$ are coordinates on $\mathbb{C}^{2n}$ and $z_k=x_k+iy_k$, then a chart $\phi$ induces vector fields $\partial/\partial x_k$ and $\partial/\partial y_k$ in $U$. At a point $p \in U$, the real tangent space $T_p M$ has a complex structure defined to be the real linear extension of
 \begin{equation} \label{eq:almost_complex_structure}
    J_p \frac{\partial}{\partial x_k} = \frac{\partial}{\partial y_k}, \ \ \ J_p \frac{\partial}{\partial y_k} = -\frac{\partial}{\partial x_k}.
 \end{equation} 
 It can be shown that this complex structure is independent of the choice of coordinates, using the Cauchy-Riemann equations.
\end{example}
\begin{remark}
 In general, a smoothly varying choice of complex structure on each tangent space of a real $2n$-dimensional manifold $M$ is called an almost complex structure. Not every manifold with an almost complex structure can be given an atlas making it a complex manifold with almost complex structure arising from Equation \eqref{eq:almost_complex_structure}. If this can be done, the almost complex structure is called integrable. 
\end{remark}

\begin{example} \label{ex:Hodge_star_real}
 Let $\mathscr{R}$ be a compact Riemann surface, (i.e.~a compact complex manifold of dimension 2).
 In local coordinates $z=x+iy$, any real one-form can be expressed as $\beta = a(z) dx + b(z) dy$.  
 The Hodge star operator on one-forms is defined to be the complex linear extension of  
 \[  \ast dx = dy, \ \ \ \ast dy = -dx.   \]
 By the Cauchy-Riemann equations, this is coordinate-independent.
 A real or complex one-form $\beta$ is said to be harmonic if it is closed and co-closed, that is, if $d \beta=0$ and $d \ast \beta =0$ respectively. Equivalently, in local coordinates $\beta(z) = h(z) dz$ where $h(z)$ is harmonic.
 
 By the Hodge theorem, every deRham cohomology class on $\mathscr{R}$ has a harmonic representative. So the set of real harmonic one-forms  $\mathcal{A}^{\mathbb{R}}_{\mathrm{harm}} (\mathscr{R})$ on $\mathscr{R}$ is a real $2g$-dimensional vector space.  Define the pairing  
 \begin{equation} \label{eq:pairing_oneforms} 
     g(\beta,\gamma)  =  \iint_{\mathscr{R}}  \beta \wedge \ast {\gamma}. 
 \end{equation}
 It is easily checked that this is an inner product 
 on $\mathcal{A}^{\R}_{\mathrm{harm}}(\mathscr{R})$. Define also 
  \[ \omega(\beta,\gamma) =  \iint_{\mathscr{R}} \beta \wedge \gamma. \]

  With $H=\mathcal{A}^{\mathbb{R}}_{\mathrm{harm}} (\mathscr{R})$, $(g,\ast,\omega)$ is a compatible triple. 
\end{example}
\begin{remark} 
    In the previous example, the restriction of the Hodge star operator to a single tangent space is the dual of the almost complex structure given in  Example \ref{ex:complex_manifold}. 
\end{remark}
\begin{exercise} \label{exer:hodge_well_defined_symplectic} Let $\mathscr{R}$ and $\omega$ be as in Example \ref{ex:Hodge_star_real}. 
    Let $H^{1}_{\text{dR}}(\mathscr{R})$ denote the deRham cohomology space of smooth closed one-forms modulo smooth exact one-forms on $\mathscr{R}$. Given $[\beta],[\gamma] \in H^{1}_{\text{dR}}(\mathscr{R})$, let $\hat{\beta},\hat{\gamma}$ be the harmonic representatives of their respective equivalence classes. Show that for arbitrary representatives $\beta$, $\gamma$ of $[\beta]$, $[\gamma]$
    \[  \omega(\beta,\gamma) = \omega(\hat{\beta},\hat{\gamma}). \]
    In particular, $\omega$ is a well-defined symplectic form on $H^{1}_{\text{dR}}(\mathscr{R})$. 
\end{exercise}

Next, we introduce two examples that are of fundamental importance in both complex function theory and conformal field theory.
\begin{example}  \label{ex:real_boson_example}
 Let $\bose$ denote the space of sequences \[  \left\{ \{ a_n \}_{n \in \mathbb{Z}, n \neq 0}  \,:\, a_n \in \mathbb{C}, \ a_{-n} = \overline{a_{n}}, \ \sum_{n=1}^\infty n |a_n|^2 <\infty \right\}.   \]  
 Note that $\bose$ is a subset of $\ell^2$ (indexed doubly infinitely), and in particular the Fourier series
 \[  \sum_{n \in \Z \setminus \{0\}} a_n e^{in\theta}    \]
 converges almost everywhere on $\mathbb{S}^1$.   
 The space $\bose$ can be identified with the real homogeneous Sobolev space $\dot{H}^{1/2}_{\mathbb{R}}(\mathbb{S}^1)$. 
 For the purpose of this example, we define $\dot{H}^{1/2}_{\mathbb{R}}(\mathbb{S}^1)$ to be the subset of $L^2_{\mathbb{R}}(\mathbb{S}^1)$ whose Fourier coefficients are in $\bose$.
 
 This is a Hilbert space with respect to the following inner product: given
 \[  \{ a_n \}_{n \in \mathbb{Z} \backslash \{0\}} , \ \ \    \{ b_n \}_{n \in \mathbb{Z} \backslash \{0\}}   \]
 we have
 \begin{equation} \label{eq:g_in_Hb} 
  g( \{ a_{n} \}, \{ b_{n} \} ) = 2 \operatorname{Re} \sum_{n=1}^{\infty} n a_{n} \overline{b_{n}} = \sum_{n \in \Z \setminus \{0\}} |n| a_n \overline{b_n}. 
  \end{equation}
 Note that this is {\it not} the inner product in $L^2(\mathbb{S}^1)$. 
 We also have the symplectic form 
 \begin{equation} \label{eq:omega_in_Hb} 
  \omega( \{ a_{n} \}, \{ b_{n} \} ) = 2 \operatorname{Im} \sum_{n=1}^{\infty} n a_{n} \overline{b_{n}} = -i \sum_{n \in \Z \setminus \{0\}}  n a_n b_{-n} .  
 \end{equation} 
 
 The Hilbert transform $J:\bose \rightarrow \bose$ is defined to be the linear extension of the map
 \[  J(e^{i n\theta}) = -i \, \mathrm{sgn}(n) e^{i n \theta} \]
 where $\mathrm{sgn}(n)=n/|n|$ for $n \neq 0$.
 With these choices $(g,J,\omega)$ is a compatible triple on $\bose$. 
\end{example} 
\begin{exercise}  \label{exer:smooth_symplectic_Hb}
    In the previous example, show that if $f_1$ and $f_2$ are functions on $\mathbb{S}^1$ corresponding to $\{a_n \}$, $\{b_n \}$ respectively which happen to be smooth, then we can write the symplectic form as 
 \begin{equation*}
     \omega( \{ a_{n} \}, \{ b_{n} \} )= \frac{1}{2 \pi} \int_{\mathbb{S}^1} f_1 \,df_2.
 \end{equation*}
\end{exercise}
\begin{remark}
 Fourier series arising from elements of $\bose$ are precisely the set of functions on $\mathbb{S}^1$ arising as boundary values of harmonic functions $h$ on $\mathbb{D}=\{ z: |z|<1 \}$ in the real homogeneous Dirichlet space 
 \[   \dot{\mathcal{D}}_{\mathbb{R},\mathrm{harm}}(\mathbb{D}) = \left\{ h:\mathbb{D} \rightarrow \mathbb{R} \ \text{harmonic} \,:\, h(0)=0, \  \iint_\mathbb{D} \left|\nabla h \right|^2   <\infty. \right\}. \]
 The function $h$ is obtained from the Fourier series by replacing $e^{i n \theta}$ with $z^n$ for $n>0$, and with $\bar{z}^n$ for $n <0$.
 
 Equivalently, the Dirichlet space is the set of anti-derivatives of $L^2$ harmonic one-forms on $\mathbb{D}$, which vanish at $0$. 
 In summary 
 \[ \bose_{\mathbb{R}} \simeq \dot{H}^{1/2}_{\mathbb{R}}(\mathbb{S}^1) \simeq \dot{\mathcal{D}}_{\mathbb{R},\mathrm{harm}}(\mathbb{D}). \] 
\end{remark}
\begin{example}   \label{ex:fermion_real_circle} 
Let $\fermi$ be the space of complex-valued sequences
\begin{equation*}
     \fermi \defeq \left\{ \{ a_n \}_{n \in \mathbb{Z}}: \sum |a_n|^2 <\infty, a_{-n} = \overline{a_n}  \right\}.
\end{equation*}
     We equip $\fermi$ with a compatible triple, given by $(g,J,\omega)$
    \begin{align*}
        g( \{ a_{n} \}, \{ b_{n} \} ) &= 2 \operatorname{Re} \sum_{n=0}^{\infty}  a_{n} \overline{b_{n}}, & J\{ a_{n} \} &= \{ i a_{n} \}, & \omega( \{ a_{n} \}, \{ b_{n} \} ) &= 2 \operatorname{Im} \sum_{n=0}^{\infty} a_{n} \overline{b_{n}}.
    \end{align*}
    We note that $\fermi$ is a (real) Hilbert space with respect to $g$.
    We wish to view elements $f = \{ a_{n} \}$ of $\fermi$ as \emph{real}-valued functions on the circle, by setting
    \begin{equation*}
        f(z) = \sum_{n=0}^{\infty} a_{n}z^{n+1/2} + \overline{a_{n}z^{n+1/2}},
    \end{equation*}
    where $z^{1/2}$ is the following choice of square root on $S^{1}$: $e^{i \theta} \mapsto e^{i \theta/2}$ for $0 \leqslant \theta < 2\pi$.
    This identifies $\fermi$ with the space of real-valued square-integrable functions on the circle.
    $\fermi$ arises naturally as a description of the space of $L^2$-sections of the odd spinor bundle on the circle (\cite[Section 2]{kristel_fusion_2019}). Equivalently, it can be identified with a space of $L^2$-half-densities on the circle. 
    
    The factor $z^{1/2}$ is necessary to identify these spaces as a space of functions.  
    However, this identification has the feature that even analytically well-behaved elements of these spaces (e.g.~$a_{i} = \delta_{i0}$) are represented by discontinuous functions.
    In fact, because the odd spinor bundle on the circle is non-trivial, there is no $C^{\infty}(S^{1})$-equivariant way to identify its smooth sections with the smooth functions on the circle (see e.g.~the Serre-Swan theorem).
\end{example} 

We say that a pairing $\langle - , - \rangle$ on a complex vector space $H$ is {\it sesquilinear} if it is complex linear in the first entry, conjugate linear in the second entry, and $\left<v,w\right>=\overline{\left<w,v\right>}$ for all $v,w \in H$. 

\begin{remark} \label{re:non-doubled_complexification}
	Suppose that a compatible triple $(g,J,\omega)$ is given, and suppose that $H$ is a Hilbert space with respect to $g$.
	We denote by $H_{J}$ the complex vector space $H$ where complex multiplication is given by $Jv = iv$.
	Show that the form
	\begin{equation*}
		\langle v, w \rangle = g(v,w)   -  i \omega(v,w) = g(v,w)  - i g(Jv,w).
	\end{equation*}
    is sesquilinear.
	This turns $H_{J}$ into a complex Hilbert space.  
\end{remark}

In the next section, we consider instead the complexification of a space $H$ with compatible triple.

\section{Polarizations}\label{sec:Polarizations}
 In this section, we introduce the notion of polarization.
 This is a decomposition of the complexification of the Hilbert space into conjugate subspaces.
 We consider two kinds of polarization: those in which the subspaces are symplectic Lagrangian subspaces, and those in which the subspaces are orthogonal. The symplectic decomposition is assumed to satisfy an additional positivity condition (the asymmetry between the two cases is an artefact of the fact that symplectic forms are only required to be non-degenerate, whereas an inner product must also be positive definite). 
 We refer to these as positive symplectic polarizations and orthogonal polarizations respectively. The terminology is not standard, but was chosen to easily keep track of the two varieties of polarization.
 For example, in the symplectic literature, the term positive polarization is used.  

 We outline the basic theory, with examples illustrating a few of the ways that they arise.  The main aim of this section is to show that a positive symplectic polarization defines a unique complex structure and complex inner product resulting in a compatible triple on the original real Hilbert space. Similarly, an orthogonal polarization defines a unique complex structure and symplectic form which restrict to a compatible triple on the real space. 
 Thus in some sense the two concepts coincide. 
 
 However, the two perspectives are quite different when it comes to deformations of structures. In symplectic geometry the situation often arises that the symplectic form is fixed, and the inner product and complex structure vary; while in spin geometry, the situation where the inner product is fixed but the complex structure and symplectic form vary arises. We will explore those in subsequent sections.

 We proceed as follows. First, we fix a compatible triple, and show how it naturally specifies a decomposition which is both a positive symplectic and an orthogonal polarization. Then, we show how each of the two types of polarization naturally defines a compatible triple.  

Let us write $H_{\C}$ for the complex Hilbert space $H \otimes_{\R} \C$ (not to be confused with $H_{J}$).
The fact that $H_{\C}$ arises as the complexification of the real Hilbert space $H$ is witnessed by the map $\alpha: H_{\C} \rightarrow H_{\C}$ defined to be the real-linear extension of the map $v \otimes \lambda \mapsto v \otimes \overline{\lambda}$.
The map $\alpha$ is a \emph{real structure}: it is a conjugate-linear, isometric involution.
In applications, it is often more natural to start from a complex Hilbert space, which is then equipped with a real structure.
The real Hilbert space can then be recovered as the vectors that are fixed by the real structure.

If $T:H \rightarrow H$ is a bounded linear operator, then we write $T:H_{\C} \rightarrow H_{\C}$ for its complex-linear extension.
A complex-linear operator $T:H_{\C} \rightarrow H_{\C}$ corresponds to a linear operator $H \rightarrow H$ if and only if it commutes with $\alpha$.
We extend $g$ to $H_{\C}$ sesquilinearly, i.e.~
\begin{equation*}
    g(v \otimes \lambda, w \otimes \mu) = \lambda \overline{\mu} g(v,w),
\end{equation*}
and extend $\omega$ to $H_{\C}$ bilinearly. We then have that
\begin{align*}
    g(v,w) &= \omega(v, J \alpha w), & \omega(v,w) &= g(Jv, \alpha w).
\end{align*}
Any complex structure $J$ on $H$ induces a (complex direct-sum) decomposition $H_{\C} = L^{+} \oplus L^{-}$, by setting
\begin{align} \label{eq:standard_pol_def}
	L^{+} &= \{ v \in H_{\C} \, : \, Jv = iv \}, & L^{-} &= \{ v \in H_{\C} \, : \, Jv=-iv \}.
\end{align}
We make the (somewhat trivial) observation that $L^{\pm}$ are \emph{closed} subspaces because they can be realized as $L^{\pm} = \ker( J \mp i)$; this observation will be important later.
\begin{remark}\label{rem:WhyIsAlphaL+EqualToL-}
The decomposition $H_{\C} = L^{+} \oplus L^{-}$ has the property that $\alpha(L^{+}) = L^{-}$.
Indeed, let $v = x \otimes 1 + y \otimes i \in L^{+}$, we compute 
    \begin{equation*}
        - y \otimes 1 + x \otimes i = iv = Jv = J(x \otimes 1 + y \otimes i) = Jx \otimes 1 + Jy \otimes i,
    \end{equation*}
    and thus $Jx = -y$ and $Jy = x$.
    It follows that
    \begin{equation*}
        J\alpha v = J(x \otimes 1 - y \otimes i) = -y \otimes 1- x \otimes i = -i (x \otimes 1 - y \otimes i) = -i \alpha v,
    \end{equation*}
    whence $\alpha v \in L^{-}$, and thus $\alpha(L^{+}) \subseteq L^{-}$.
    Similarly, one proves $\alpha(L^{-}) \subseteq L^{+}$, which implies $L^{-} \subseteq \alpha(L^{+})$, and thus $\alpha (L^{+}) = L^{-}$.

Conversely, suppose that we are given a direct-sum decomposition $H_{\C} = W^{+} \oplus W^{-}$.
Write $P_{W^{\pm}}$ for the projection of $H_{\C}$ onto $W^{\pm}$ along $W^{\mp}$.
From the above, we know that for $J_{W} = i (P_{W^{+}} - P_{W^{-}})$ to restrict to a complex structure, a necessary condition is that $\alpha(L^{+}) = L^{-}$.
One easily verifies that this is also a sufficient condition, by computing $[\alpha,J_{W}]$.
\end{remark}

Fix now a compatible triple $(g,J,\omega)$ on $H$.
A subspace $L \subset H$ is called \emph{isotropic} (with respect to $\omega$) if $\omega$ is identically zero on $L \times L$.
A subspace $L \subset H$ is \emph{Lagrangian} if it is isotropic, and there exists another isotropic subspace $L' \subset H$ such that $H = L \oplus L'$.
Because $H$ is taken to be a Hilbert space, we have the following characterization of Lagrangian subspace, taken from \cite[Proposition 5.1]{weinstein_symplectic_1971}.
\begin{lemma}\label{lem:maximalMeansLagrangian}
    If $L \subset H$ is a maximal isotropic subspace, then it is Lagrangian.
\end{lemma}
The compatibility of the triple $(g,J,\omega)$ is reflected in the decomposition $H = L^{+} \oplus L^{-}$ as follows.
\begin{lemma} \label{le:standard_polarization_all_properties}
The decomposition $H_{\C} = L^{+} \oplus L^{-}$ induced by $J$ has the following properties:
    \begin{itemize}
        \item it is orthogonal with respect to the sesquilinear extension of $g$;
        \item it is Lagrangian with respect to the complex-bilinear extension of $\omega$.
    \end{itemize}
\end{lemma}
\begin{proof}
    
    We prove the orthogonality claim.
    Let $v \in L^{+}$ and $w \in L^{-}$.
    We then have that $g(v,w) = g(Jv,Jw) = g(iv,-iw) = -g(v,w)$, whence $g(v,w) = 0$.
    The claim then follows from the fact that $L^{+}$ and $L^{-}$ are closed subspaces of $H_{\C}$ (and thus $(L^{\pm})^{\perp \perp} = L^{\pm}$).
    The claim that the decomposition is Lagrangian follows similarly.
\end{proof}

If $T: H_{\C} \rightarrow H_{\C}$ is either a complex-linear or conjugate-linear operator, then we write
\begin{equation*}
    T = \begin{pmatrix}
        a & b \\
        c & d
    \end{pmatrix} \begin{array}{c} L^+ \\ \oplus \\ L^- \end{array} \rightarrow \begin{array}{c} L^+ \\ \oplus \\ L^- \end{array} 
\end{equation*}
for the corresponding decomposition; where $a,b,c,d$ are complex-linear, resp.~conjugate-linear operators. 
It follows from \cref{le:standard_polarization_all_properties} that with respect to the decomposition $H_{\C} = L^{+} \oplus L^{-}$ we have
\begin{equation*}
    \alpha = \begin{pmatrix}
        0 & \left. \alpha \right|_{L^{-}} \\
        \left. \alpha \right|_{L^+} & 0
    \end{pmatrix}.
\end{equation*} 
A straightforward computation shows that the operator $T$ commutes with $\alpha$ if and only if $\alpha a \alpha = a$, and $\alpha d \alpha = d$, and $\alpha b = c\alpha$.

In the sequel, we will have to compare the operator $T$ to the operator $\alpha T \alpha$.
In some sources, the notation $\alpha T \alpha = \overline{T}$ is used.
This notation is justified by the observation that if $T_{ij} \in \C$ is some matrix component of $T$, with respect to a basis $\{e_{i} \}$ of $\alpha$-fixed vectors, then $(\alpha T \alpha)_{ij} = \overline{T_{ij}}$.
We eschew this notation, because we will usually be working with bases consisting of vectors that are not $\alpha$-fixed; in particular, no non-zero element of $L^{\pm}$ is $\alpha$-fixed.

\begin{exercise} \label{exer:standard}
 Consider the vector space $H = \R^{2n}$.
 Let $\{e_{i}\}_{i=1,...,2n}$ be the standard basis, and write $x_{i} = e_{2i-1}$ and $y_{i} = e_{2i}$.
 Let $g$ be the standard inner product, and define $J$ by $J x_{i} = y_{i}$ and $J y_{i} = - x_{i}$.
 Determine the symplectic form $\omega$ such that $(g, J, \omega)$ is a compatible triple.
 Then, determine the splitting $H \otimes_{\R} \C = \C^{2n} = L^{+} \oplus L^{-}$.
\end{exercise}

\begin{exercise} \label{exer:two_models_complexify}
   Let $H$ be a real Hilbert space with complex structure $J$.  Let $H_J$ be as in Remark \ref{re:non-doubled_complexification}.  Show that 
   \begin{align*}
    \Psi: H_J \rightarrow L^+, \quad
     v \mapsto \frac{1}{\sqrt{2}} \left( v-i J v \right) 
   \end{align*}
   is a complex vector space isomorphism. 
\end{exercise} 
\begin{example}
  Let $M$ be a complex manifold. For a point $p \in M$, let $J_p$ be as in Example \ref{ex:complex_manifold}.  The induced decomposition of the complexified tangent space 
  \[  \mathbb{C} \otimes T_p M = L^+ \oplus L^- \] 
  is often denoted 
  \[  L^+=T^{(1,0)}_p M, \ \ L^- = T^{(0,1)}_p M. \]
  $T^{(1,0)}_pM$ is called the holomorphic tangent space. 

  By Exercise \ref{exer:two_models_complexify}, $\Psi$ is an isomorphism between the real tangent space $T_pM$ and the holomorphic tangent space $T^{(1,0)}_pM$.  If $(z_1,\ldots,z_n)$ are local coordinates on $M$ near $p$, then 
  \begin{equation*}
    T^{(1,0)}_p M = \mathrm{span} \left\{ \frac{\partial}{\partial z_1},\ldots,\frac{\partial}{\partial z_n} \right\}.  \qedhere
  \end{equation*}
\end{example}  
\begin{exercise} \label{ex:Hermitian}
  Now assume that $(g,J,\omega)$ is a compatible triple for $H$ in Exercise \ref{exer:two_models_complexify}.  Let $g_+$ denote the restriction to $L^+$ of the sesquilinear extension of $g$.  Show that $\Psi:(H_J,\langle,\rangle) \rightarrow (L^+,g_+)$ is an isometry; that is, $g_+(\Psi(v),\Psi(w))=\langle v,w \rangle$ for all $v,w \in H$. 
\end{exercise}
\begin{remark}  In the previous exercise, $\langle -,- \rangle$ and $g^+$ are both often referred to as a Hermitian metric. The data $J$ and a metric $g$ such that $g(Jv,Jw)=g(v,w)$ for all $v,w$ are sufficient to recover the compatible triple $(g,J,\omega)$ by \cref{exer:two_of_three_gives_compat}.
\end{remark}
\begin{example} 
  Let $M$ be a finite-dimensional complex manifold. A Hermitian metric on $M$ is a Hermitian metric on $T_p M$ for each $p \in M$ (or equivalently $T^{(1,0)}_pM$), which varies smoothly with $p$. 
  If the associated form $\omega$ is closed (so that $M$ is also a symplectic manifold), then $M$ is called a K\"ahler manifold. 
\end{example}

\begin{example} \label{ex:polarization_compact_surface}
 Let $\mathscr{R}$ be a compact Riemann surface, and $H= \mathcal{A}^{\mathbb{R}}_{\mathrm{harm}}(\mathscr{R})$ be the vector space of real harmonic one-forms on $\mathscr{R}$, and let $J = \ast$, $\omega$, and $g$ be as in Example \ref{ex:Hodge_star_real}.
 Then $H_{\mathbb{C}}$ is the vector space of complex-valued harmonic one forms, which we denote by  $\mathcal{A}_{\mathrm{harm}}(\mathscr{R})$.
 The dimension of $\mc{A}_{\mathrm{harm}}(\mathscr{R})$ over $\C$ is twice the genus of $\mathscr{R}$.
 The real structure $\alpha$ is just complex conjugation.  

 Denote the set of holomorphic one-forms on $\mathscr{R}$ by $\mathcal{A}(\mathscr{R})$, and the anti-holomorphic one-forms by $\overline{\mathcal{A}(\mathscr{R})}$.
 We have the direct sum decomposition 
 \[  \mathcal{A}_{\mathrm{harm}}(\mathscr{R}) = \mathcal{A}(\mathscr{R}) \oplus \overline{\mathcal{A}(\mathscr{R})}. \]
 Explicitly, every complex harmonic one-form $\beta$ has a unique decomposition $\beta = \gamma + \overline{\delta}$
 where $\gamma,\delta \in \mathcal{A}(\mathscr{R})$.  
 Locally these are written $\gamma= a(z) dz, \overline{\delta} = \overline{b(z)} d\bar{z}$ where $a$ and $b$ are holomorphic functions of the coordinate $z$. 
 It is easily checked that $J \gamma = -i \gamma$ and $J \overline{\delta}= i \overline{\delta}$.  So 
 \[   L^-= \mathcal{A}(\mathscr{R}), \ \ \ L^+ = \overline{\mathcal{A}(\mathscr{R})}.  \]

 The corresponding sesquilinear extension of the inner product $g$ in Example \ref{ex:Hodge_star_real} is 
 \begin{equation}
     g(\beta_1,\beta_2)  =  \iint_{\mathscr{R}}  \beta_1 \wedge \ast \overline{\beta_2}. 
 \end{equation}
\end{example} 

\begin{example}[Sobolev-$\nicefrac{1}{2}$ functions on the circle]\label{ex:bosonExample}
    Let $\bose$ be as in Example \ref{ex:real_boson_example}. 
    Treating $\bose$ as a set of functions on the circle (i.e. as $\dot{H}^{1/2}_{\mathbb{R}}(\mathbb{S}^1)$), its complexification $\bose_{\mathbb{C}}$ is the set of complex-valued functions in $L^2(\mathbb{S}^1)$ whose Fourier series 
    \[ \sum_{n \in \mathbb{Z} \backslash \{0\}} a_n e^{i n \theta} \]
    satisfy the stronger condition
    \[   \sum_{n \in \mathbb{Z} \backslash \{0\}} |n| |a_n|^2<\infty.  \]
    that is, the Sobolev-$1/2$ space $\dot{H}^{1/2}(\mathbb{S}^1)$. Equivalently, one removes the condition $a_{-n} = \overline{a_n}$ in the sequence model of $\bose$: 
    \[ \bose_{\mathbb{C}} = \left\{ \{a_n\}_{n \in \mathbb{Z} \backslash \{0\}} \,:\, \sum_{n \in \mathbb{Z} \backslash \{0\}} |n| |a_n|^2<\infty \right\}. \]
    In this case, the real structure is given by complex conjugation of the function on $\mathbb{S}^1$; equivalently, 
    \[   \alpha(\{ a_n \}) = \{ \overline{a_{-n}} \}. \]

    We then see that $L^{-}$ is given by the sequences $\{ a_{n} \}$ such that $a_{n} = 0$ for $n < 0$ and $L^+$ is given by the sequences such that $a_n = 0$ for $n>0$.  
    Equivalently, $L^{-}$ consists of those functions that extend to holomorphic functions on the unit disk while elements of $L^+$ extend anti-holomorphically.
\end{example}
\begin{exercise} \label{exer:standard_polarization_Hb}
    In the previous example, show that 
    \[  \omega(\{a_n\},\{b_n\}) = -i \sum_{n=-\infty}^\infty  n a_n   b_{-n}  \]
    and 
    \[ g(\{a_n\},\{b_n\}) = \sum_{n=-\infty}^\infty |n| a_n \overline{b_n} \]
    and verify $g(v,w)=\omega(v,J \alpha w)$. 
    Verify that $L^+ \oplus L^-$ is both Lagrangian and orthogonal. 
\end{exercise}
\begin{remark}
    The polarization in the previous example corresponds to the decomposition of complex harmonic functions in the homogeneous Dirichlet space on $\mathbb{D}$ into their holomorphic and anti-holomorphic parts. Defining 
    \[  \dot{\mathcal{D}}(\mathbb{D}) = \left\{ h:\mathbb{D} \rightarrow \mathbb{C} \ \mathrm{holomorphic} \,:\, \iint_{\mathbb{D}} |h'|^2 <\infty \right\}   \]
    and   
     \[   \dot{\mathcal{D}}_{\mathrm{harm}}(\mathbb{D}) = \left\{ h:\mathbb{D} \rightarrow \mathbb{R} \ \text{harmonic} \,:\, h(0)=0, \  \iint_\mathbb{D} \left| \frac{\partial h}{\partial z} \right|^2 +\left| \frac{\partial h}{\partial \bar{z}} \right|^2  <\infty. \right\} \]
     we have the decomposition
     \[ \dot{\mathcal{D}}_{\mathrm{harm}}(\mathbb{D})= \dot{\mathcal{D}}(\mathbb{D}) \oplus \overline{\dot{\mathcal{D}}(\mathbb{D})}. \]
    Extending elements of $\bose_{\mathbb{C}}$ harmonically to $\mathbb{D}$ we can identify $L^-\simeq \dot{\mathcal{D}}(\mathbb{D})$, $L^+ \simeq \overline{\dot{\mathcal{D}}(\mathbb{D})}$.
\end{remark}

\begin{example}[Square-integrable functions on the circle]\label{ex:fermionExample}   
   Let $\fermi$ be as in Example \ref{ex:fermion_real_circle}. 
    The complexification of $\fermi$ can be identified with the complex-valued square-integrable functions on the circle, which we in turn identify (by taking Fourier coefficients) with the space of complex-valued sequences $\{a_{n} \}$ indexed by $\mathbb{Z}$ such that the sum
    \begin{equation*}
        \sum_{n} |a_{n}|^{2}
    \end{equation*}
    converges.
    We then see that $L^{+}$ is given by the sequences $\{ a_{n} \}$ such that $a_{n} = 0$ for $n < 0$.
    Equivalently, $L^{+}$ consists of those functions $f$, such that $z \mapsto f(z)z^{-1/2}$ extends to a holomorphic function on the unit disk.
\end{example}
\begin{remark}
    The set of holomorphic extensions $h(z)$ to the unit disk of functions $f(z)z^{-1/2}$ in the previous example is precisely the Hardy space. This follows from a classical characterization of the Hardy space as  the class of holomorphic functions on the disk, whose non-tangential boundary values are in $L^2(\mathbb{S}^1)$.
    On the other hand $L^-$ consists of those $f(z) z^{1/2}$ which extend anti-holomorphically into the unit disk. These extensions are precisely the conjugate of the Hardy space.    
\end{remark}

It turns out that the space of deformations of either $g$ or $\omega$ has a rich geometric structure.  In the rest of this section, we define two notions of polarization - those arising from a fixed inner product (orthogonal polarizations), and those arising from a fixed symplectic form (positive symplectic polarizations). Both lead to a compatible triple. 

\begin{definition} 
Let $(H,g)$ be a real Hilbert space and let $H_{\mathbb{C}}$ be its complexification with real structure $\alpha$.
Extend $g$ sesquilinearly to $H_{\mathbb{C}}$.
An \emph{orthogonal polarization} is an orthogonal decomposition
\[  H_{\C} = W^{+} \oplus W^{-} \] 
such that $\alpha (W^{+}) = W^{-}$.
\end{definition} 

One can associate a complex structure and bilinear form to an orthogonal polarization. 
\begin{definition}
   Let $(H,g)$ be a real Hilbert space and let $H_{\mathbb{C}}$ be its complexification with real structure $\alpha$. Let $H_{\C} = W^{+} \oplus W^{-}$ be an orthogonal polarization.
   Let $P_{W}^{\pm}: H_{\C} \rightarrow W^{\pm}$ be the orthogonal projections.
   The complex structure associated to the polarization is $J_{W} = i(P_{W}^{+} - P_{W}^{-})$, and the bilinear form on $H_{\mathbb{C}}$ associated to the polarization is
    \begin{equation*}
        \omega_W(v,w) = g(J_W v,w). \qedhere
    \end{equation*}
\end{definition}
Observe that by the reasoning in Remark \cref{rem:WhyIsAlphaL+EqualToL-}, the condition that $\alpha(W^{+}) = W^{-}$ is both necessary and sufficient for $J_{W}$ to restrict to $H$.

An orthogonal polarization defines a compatible triple. 
\begin{proposition}\label{lem:SubspaceToSymplectic} Let $H_{\mathbb{C}}= W^+ \oplus W^-$ be an orthogonal polarization. 
    The complex structure $J_{W}$ and symplectic form $\omega_W$ associated to this polarization restrict to $H$, and $(g,J_{W},\omega_{W})$ is a compatible triple.
    Moreover, $W^{+}$ and $W^{-}$ are Lagrangian with respect to $\omega_{W}$.
\end{proposition}
\begin{proof}
    First, it is clear that $J_{W}^{2} = - \1$.
    Now, let $v = (x^{+},x^{-}) \in W^{+} \oplus W^{-}$ be arbitrary.
    We then have $\alpha J_{W} v = \alpha (ix^{+},-ix^{-}) = (i \alpha x^{-}, - i \alpha x^{+}) = J_{W} \alpha v$.
    Observe that $H \subseteq H_{\C}$ can be identified as the $+1$ eigenspace for $\alpha$; it follows that if $v \in H$, then $\alpha J_{W}v = J_{W} \alpha v = J_{W}v$ and thus $J_{W}v \in H$.
    It immediately follows that $\omega_W(v,w)=g(J_Wv,w)$ is real on $H$. 
    
    It remains to be shown that $(g,J_{W},\omega_{W})$ is a compatible triple. 
    Of course, if $\omega_{W}$ is symplectic, then it is compatible.
    That $\omega_{W}$ is skew-symmetric follows from the computation $\omega_{W}(v,w) = g(J_{W}v,w) = -g(v,J_{W}w) = -\omega_{W}(w,v)$.
    The map $\ph_{\omega_{W}}: H \rightarrow H^{*}$ is given by $\ph_{g}J_{W}$, and is thus an isomorphism.
    
    It is straightforward to show that $W^{+}$ and $W^{-}$ are isotropic, and the assumption that $H_{\C} = W^{+} \oplus W^{-}$ then implies that they are both Lagrangian.
\end{proof}
Let $\Pol^{g}(H)$ denote the set of  those closed subspaces $W^{+} \subset H_{\C}$ such that $H_{\C} = W^{+} \oplus \alpha(W^{+})$ as an orthogonal direct sum.
Observe that by \cref{lem:SubspaceToSymplectic} and \cref{le:standard_polarization_all_properties} we may view $\Pol^{g}(H)$ as parameterizing the set of deformations of $\omega$ (compatible with $g$).

Next, we define positive symplectic polarizations.   
\begin{definition}  \label{de:Kahler_pol}
    Let $(H,\omega)$ be a real vector space with symplectic form $\omega$, and let $H_{\mathbb{C}}$ be its complexification with real structure $\alpha$. Extend $\omega$ complex bilinearly to $H_{\mathbb{C}}$.

  A \emph{positive symplectic polarization} is a Lagrangian decomposition
\[  H_{\C} = W^{+} \oplus W^{-} \] 
such that $\alpha (W^{+}) = W^{-}$ and $-i \omega(v,\alpha{w})$ is a positive-definite sesquilinear form on $W^+$.
\end{definition}

In fact, this implies that $g_+(v,w) := -i\omega(v,\alpha w)$ is a strong inner product, which by Proposition \ref{pr:StrongEquivalence} makes $W_+$ a Hilbert space. To see this, observe that $\ph_{\omega}: H_{\C} \rightarrow H_{\C}^{*}$ restricts to an isomorphism $\ph_{\omega}: W^{+} \rightarrow (W^{-})^{*}$.
A direct computation shows that $\ph_{g_+}: W^{+} \rightarrow (W^{+})^{*}$ is given by $\ph_{g_+} = -i \alpha^{*} \ph_{\omega}$, whence $\ph_{g_+}$ is an isomorphism. 
It follows that $g_+$ is a strong inner product if and only if it is positive-definite.

\begin{exercise} \label{exer:Kpol_other_half_pos_def}
    Show that if $W^{+} \oplus W^{-} = H_{\C}$ is a positive symplectic polarization, then the pairing $i\omega(v,\alpha{w})$ is a positive-definite sesquilinear form on $W^-$, such that $W^-$ is a Hilbert space with respect to this pairing.
\end{exercise}

As with orthogonal polarizations, a positive symplectic polarization is enough to define the other two pieces of data.
\begin{definition}\label{def:ComplexPlusSesqui}
   Let $(H,\omega)$ be a real Hilbert space, equipped with symplectic form $\omega$.
   Let $H_{\C} = W^{+} \oplus W^{-}$ be a symplectic polarization; let $P_{W}^{\pm}: H_{\C} \rightarrow W^{\pm}$ be the projections with respect to this decomposition.
   The complex structure associated to the polarization is $J_{W} =i (P_{W}^{+} - P_{W}^{-})$, and the sesquilinear form on $H_{\mathbb{C}}$ associated to the polarization is 
    \begin{equation*}
        g_W(v,w) = \omega(v,J\alpha{w}). \qedhere
    \end{equation*}
\end{definition}
The condition that the symplectic polarization $H_{\C} = W^{+} \oplus W^{-}$ is positive is exactly the condition that is required to guarantee that $(g_{W},J_{W},\omega)$ is compatible.
\begin{proposition}\label{pr:SubspaceToInner} Let $H_{\mathbb{C}}= W^+ \oplus W^-$ be a positive symplectic polarization. 
    The associated complex structure $J_{W}$ and sesquilinear form $g_W$ restrict to a complex structure and inner product on $H$.
    Moreover, $g_W$ is a positive definite sesquilinear form on $H_{\mathbb{C}}$, with respect to which $H_{\mathbb{C}}$ is a Hilbert space, and  $(g_W,J_W,\omega)$ is a compatible triple. 
\end{proposition}
\begin{proof}
    It is clear that $J_{W}^{2} = - \1$.
    Sesquilinearity of $g_W$ follows immediately from the definition. 
    let $g_{W,\pm}$ denote the restrictions of $g_W$ to $W^\pm$. We have 
    \begin{align*}
        g_{W,+}(v,w) & = -i\omega(v,\alpha w) \ \ \  v,w \in W^+ \\
        g_{W,-}(v,w) & = i \omega(v,\alpha w) \ \ \ \   v,w \in W^-.
    \end{align*} 
    It is easily checked that $W^+$ and $W^-$ are orthogonal with respect to $g_W$, so we have
    \[ (H_{\C},g_W) = (W^+,g_{W,+}) \oplus (W^-,g_{W,-})  \]
    and both summands are Hilbert spaces by Exercise \ref{exer:Kpol_other_half_pos_def}. 
    
    The fact that $\alpha(W^{+}) = W^{-}$ implies that $J_{W}$ commutes with $\alpha$, which in turn implies that $J_{W}$ restricts to $H$. To see that $g_W$ is real on $H$, observe that any $v,w\in H$ can be written $(v'+\alpha{v'})/2$, $(w' + \alpha{w'})/2$ where $v',w' \in W^+$. 
    Using the definition $g_W(v,w)=\omega(v,J\alpha{w})$ and the fact that $J$ commutes with complex conjugate, we have
    \[  g_W(v'+\alpha{v'},w'+\alpha{w'})= g_W(v',w') + g_W(\alpha{v'},\alpha{w'}) \]
    which is real. Finally, restricting to $v,w \in H$, we have 
    \[ g_W(v,w) = \omega(v,Jw)  \]
    so $(g_W,J_W,\omega)$ is a compatible triple on $H$.
\end{proof}

Denote the set $W^+$ such that $H_{\mathbb{C}}=W^+ \oplus \alpha(W^+)$ is a positive symplectic polarization by $\Pol^\omega(H)$. By Lemma \ref{le:standard_polarization_all_properties} and Proposition \ref{pr:SubspaceToInner}, the positive symplectic polarizations can be viewed as parametrizing the deformations of $g$ (compatible with $\omega$).

Summarizing the section so far, we have seen that a vector space $H$, equipped with a compatible triple $(g,J,\omega)$, can equivalently be seen as
\begin{itemize}\itemsep-.4em
    \item A Hilbert space $(H,g)$, equipped with an orthogonal polarization of $H_{\C}$; or
    \item A symplectic vector space $(H,\omega)$, equipped with a positive symplectic polarization of $H_{\C}$.
\end{itemize}
\vspace{.4em}

We write $\GL(H)$ for the group of invertible linear transformations from $H$ to $H$.
We then define the \emph{symplectic group} of $H$ to be
\begin{equation*}
	\Sp(H) \defeq \{ u \in \GL(H) \, : \, \omega(u v,u w) = \omega(v,w) \text{ for all } v,w \in H \}.
\end{equation*}
Similarly, we define the \emph{orthogonal group} of $H$ to be
\begin{equation*}
	\O(H) \defeq \{ u \in \GL(H) \, : \, g(u v,u w) = g(v,w) \text{ for all } v,w \in H \}.
\end{equation*}

Define $u^{\pullb}g$, $u^{\pullb}J$, and $u^{\pullb}\omega$ by
\begin{align*}
    u^{\pullb}g(v,w) &= g(u^{-1}v,u^{-1}w), & u^{\pullb}Jv &= uJu^{-1}v, & u^{\pullb}\omega(v,w) = \omega(u^{-1}v,u^{-1}w).
\end{align*}
We then have the following.
\begin{proposition}
If $u \in \GL(H)$, then the triple $(u^{\pullb}g, u^{\pullb}J, u^{\pullb}\omega)$ is again compatible.
\end{proposition}
\begin{proof} It is immediate that $u^\pullb J$ is a complex structure. To prove the other two claims, 
 observe that if $\psi: H \times H \rightarrow \R$ is a bilinear pairing, then we have $\ph_{u^{\pullb}\psi}(v)(w) = u^{\pullb}\psi(v,w) = \psi(u^{-1}v,u^{-1}w) = \ph_{\psi}(u^{-1}v)(u^{-1}w)$.
    This implies that $\ph_{u^{\pullb}\psi} = (u^{-1})^*\ph_{\psi}u^{-1}$, where $u^*$ is the adjoint map $H^{*} \rightarrow H^{*}$.
    It follows that $\ph_{\psi}$ is invertible if and only if $\ph_{u^{\pullb}\psi}$ is invertible.
    Because pullback clearly preserves anti-symmetry, $u^\pullb \omega$ is a strong symplectic form. 
    The fact that $u^\pullb g$ is a strong inner product would follow from $u^{\pullb}g(v,v) \geqslant 0$ by the discussion following Definition \ref{de:Kahler_pol}. But this follows from the positivity of $g$, since $u^{\pullb}g(v,v) = g(u^{-1}v,u^{-1}v) \geqslant 0$.
\end{proof}
The groups $\GL(H)$, $\O(H)$, and $\Sp(H)$ act on $H_{\C}$ by complex-linear extension.
The decomposition $H_{\C} = L^{+}_{u} \oplus L^{-}_{u}$ corresponding to the triple $(u^{\pullb}g, u^{\pullb}J, u^{\pullb}\omega)$ is given by $L^{\pm}_{u} = u L^{\pm}$ (this follows directly from \cref{eq:standard_pol_def}).
The above guarantees that $\Sp(H)$ acts on $\Pol^{\omega}(H)$ by $(u,W^{+}) \mapsto u W^{+}$.
Indeed, if $(g_{W},J_{W},\omega)$ is the triple corresponding to $W^{+}$, and $u \in \Sp(H)$, then $(u^{\pullb}g, u^{\pullb}J,\omega)$ is the triple corresponding to $uW^{+}$, which tells us that $uW^{+} \in \Pol^{\omega}(H)$.
Similarly, one sees that $\O(H) \times \Pol^{g}(H) \rightarrow \Pol^{g}(H), (u,W^{+}) \mapsto uW^{+}$ defines an action of $\O(H)$ on $\Pol^{g}(H)$.
The \emph{unitary group} of $H$ is defined to be the intersection of the symplectic and orthogonal group, i.e.~$\U(H) \defeq \Sp(H) \cap \O(H)$.
\begin{exercise}  \label{exer:unitaries_are_same}
    Show that the unitary group $\U(H)$ is equal to the unitary group of the complex Hilbert space $H_{J}$, which is defined in the usual manner.
\end{exercise}

The following lemma tells us that, at least as sets, we may view $\Pol^{g}(H)$ and $\Pol^{\omega}(H)$ as homogeneous spaces.

\begin{lemma}\label{lem:HomogeneousSets}
    The maps $\O(H) \rightarrow \Pol^{g}(H), u \mapsto u(L^{+})$ and $\Sp(H) \rightarrow \Pol^{\omega}(H), u \mapsto u(L^{+})$ descend to bijections
    \begin{align*}
        \O(H)/\U(H) &\rightarrow \Pol^{g}(H), & \Sp(H)/\U(H) &\rightarrow \Pol^{\omega}(H) .
    \end{align*}
\end{lemma}
\begin{proof}
    An operator $u \in \O(H)$ preserves $L^{+}$ if and only if it commutes with $J$, which it does if and only if $u \in \Sp(H)$; whence the stabilizer of $L^{+} \in \Pol^{g}(H)$ is $\U(H)$.
    It thus remains to be shown that $\O(H)$ acts transitively on $\Pol^{g}(H)$.
    To this end, let $W^{+} \in \Pol^{g}(H)$ be arbitrary.
    Pick orthonormal bases $\{ l_{i} \}$ and $\{ e_{i} \}$ for $L^{+}$ and $W^{+}$ respectively.
    We obtain orthonormal bases $\{ \alpha l_{i} \}$ and $\{ \alpha e_{i} \}$ for $\alpha(L^{+})$ and $\alpha(W^{+})$ respectively.
    We let $u: H_{\C} \rightarrow H_{\C}$ be the complex-linear extension of the map that sends $l_{i}$ to $e_{i}$ and $\alpha l_{i}$ to $\alpha e_{i}$.
    It is then easy to see that $u \in \O(H)$ and $u(L^{+}) = W^{+}$.
    This proves that the map $\O(H) / \U(H) \rightarrow \Pol^{g}(H)$ is a bijection.
    
    The fact that the map $\Sp(H)/\U(H) \rightarrow \Pol^{\omega}(H)$ is a bijection is proved similarly.
\end{proof}
\begin{exercise} \label{exer:stabilizer}
    Prove that the stabilizer of $L^{+} \in \Pol^{\omega}(H)$ is $\U(H)$ and that $\Sp(H)$ acts transitively on $\Pol^{\omega}(H)$.
    Conclude that the map $\Sp(H)/\U(H) \rightarrow \Pol^{\omega}(H)$ is a bijection.
\end{exercise} 

\section{The Grassmannian of polarizations associated to a symplectic form}\label{sec:Bosons}

In this section, we fix a symplectic form, and describe the associated Lagrangian Grassmannian of positive symplectic polarizations. We will show that it is described by a space of operators called the Siegel disk, and that it is a complex Banach manifold.
We also consider the ``restricted'' Siegel disk and restricted symplectic groups. These restricted objects require the extra analytic condition of Hilbert-Schmidt-ness, which plays a central role in representation theory.  We also establish that the restricted Grassmannian is not just a Banach manifold, but a Hilbert manifold.  

Let $H$ be a real Hilbert space, and let $(g,J,\omega)$ be a compatible triple.
We are interested in deformations of $g$, while keeping $\omega$ fixed.
As explained in \cref{sec:Triples}, we can study these deformations by considering the space of symplectic polarizations in $H_{\C}$.

We denote by $P^{\pm}$ the orthogonal projection $H_{\C} \rightarrow L^{\pm}$.

\begin{proposition}\label{pr:CompatInequality} 
    Suppose that $W^{+}\subset H_{\C}$ is a Lagrangian such that $H_{\C} = W^{+} \oplus \alpha(W^{+})$ (not necessarily orthogonal); and set $W^{-} = \alpha(W^{+})$.
    Then, $W^{+}$ is a positive symplectic polarization if and only if $\| P^{+}x \| > \|P^{-} x\|$ for all $0 \neq x \in W^{+}$.
    Moreover, if this holds, then $P^{+}|_{W^{+}}:W^{+} \rightarrow L^{+}$ is a bijection.
\end{proposition}
\begin{proof}
    We recall that the decomposition $H_{\C} = W^{+} \oplus W^{-}$ allows us to define a complex structure $J_{W}$ and a sesquilinear form $g_{W}$, see \cref{def:ComplexPlusSesqui}.
    Let $x \in W^{+}$.
    We then compute
    \begin{equation*}
        g_{W}(x,x) = \omega(x, \alpha J_{W} x) = \omega(x, \alpha ix) = g(Jx, ix) = g((P^{+}-P^{-})x,x) = \| P^{+} x\|^{2} - \|P^{-}x \|^{2}.
    \end{equation*}
    We thus see that $g_{W}$ is positive definite if and only if $\|P^{+}x \|^{2} > \|P^{-}x \|^{2}$ for all $0 \neq x \in W^{+}$. 
    Repeating the proof of Proposition \ref{pr:SubspaceToInner} we see that $J_W$ and $g_W$ restrict to the real space. 
    Furthermore, we have that the restriction of $g_W$ to the real space is strong, because $\ph_{g_W}= \ph_{\omega} J_W$.   
    
    Now, suppose that $\|P^{+}x \|^{2} > \|P^{-}x\|^{2}$ for all $0 \neq x \in W^{+}$.
    We then have, for all $x \in W^{+}$
    \begin{equation*}
        \| x \|^{2} = \|P^{+} x\|^{2} + \|P^{-}x\|^{2} \leqslant 2 \|P^{+}x \|^{2},
    \end{equation*}
    whence $P^{+}|_{W^{+}}$ is bounded below, and thus injective with closed image.
    We now claim that $P^{+}W^{+}$ is dense in $L^{+}$.
    Indeed, suppose that $y \in (P^{+}W^{+})^{\perp} \subseteq L^{+}$.
    We then have, for all $x \in W^{+}$
    \begin{equation*}
        0 = g(y, P^{+}x) = g(y, x) = g(y, (P_{W}^{+} - P_{W}^{-})x) = ig(y, J_{u}x) = \omega(y, \alpha x)
    \end{equation*}
    which implies that $y$ is in the symplectic complement of $W^{-}$, which is equal to $W^{-}$.
    This means that $\alpha y \in L^{-} \cap W^{+}$, but this intersection is zero, because it is contained in the kernel of $P^{+}|_{W^{+}}$, which is zero.
    We conclude that $P^{+}|_{W^{+}}$ is a bijection, and thus invertible.  
\end{proof}
\begin{exercise}\label{exer:VerifyGraphGivesSubspace}
    Show that if $W^{+} \subset H_{\C}$ is a positive symplectic polarization, then $W^{+} = \mathrm{graph}(Z)$, where $Z = P^{-}(P^{+}|_{W^{+}})^{-1}: L^{+} \rightarrow L^{-}$.
\end{exercise}

\begin{corollary}\label{cor:GraphOperator}
    If $W^{+} \subset H_{\C}$ is a positive symplectic polarization, then the operator $Z = P^{-}(P^{+}|_{W^{+}})^{-1}$ is the unique operator such that $W^{+} = \operatorname{graph}(Z)$.
\end{corollary}
\begin{proof}
    Any operator is uniquely defined by its graph, so the only question is existence; since we have provided an expression for $Z$, all that remains to be shown is that this does trick, which is \cref{exer:VerifyGraphGivesSubspace}.
\end{proof}

Now, we would like to determine for which operators $Z: L^{+} \rightarrow L^{-}$, we have that $W_{Z} = \operatorname{graph}(Z) \in \Pol^{\omega}(H)$.
\begin{lemma}\label{lem:GraphLagrangianCondition}
    The subspace $W_{Z}$ is Lagrangian if and only if $Z = \alpha Z^{*} \alpha$.
\end{lemma}
\begin{proof}
    First, assume that $W_{Z}$ is Lagrangian.
    This means that it is in particular isotropic.
    We thus compute, for arbitrary $x,y \in L^{+}$
    \begin{align} \label{eq:lagrangian_condition_Z} 
        0 = \omega(x+Zx,y+Zy) &= \omega(x,Zy) + \omega(Zx,y)  \nonumber \\
        &= g(Jx, \alpha Zy) + g(JZx, \alpha y) \nonumber \\
        &= -g(x, (J \alpha Z + Z^{*} J \alpha) y),
    \end{align}
    and thus $Z = \alpha J Z^{*} J \alpha$.
    We have $\alpha J Z^{*} J \alpha = \alpha Z^{*} \alpha$, because the domain of $Z$ is $L^{+}$.
    
    Now, assume that $Z = \alpha J Z^{*} J \alpha$.
    The computation above then implies that $W_{Z}$ is isotropic.
    It remains to be shown that $W_{Z}$ is Lagrangian.
    To that end, suppose that $z = z^{+} + z^{-} \in L^{+} \oplus L^{-}$ satisfies $\omega(z, w) = 0$ for all $w \in W_{Z}$.
    We then have, for all $x \in L^{+}$,
    \begin{align*}
        0 = \omega(x+Zx ,z^{+} + z^{-}) &= \omega(Zx,z^{+}) + \omega(x,z^{-}) \\
        &= g(JZx,\alpha z^{+}) + g(Jx,\alpha z^{-}) \\
        &= -g(x, Z^{*}J \alpha z^{+}) - g(x, J\alpha z^{-}),
    \end{align*}
    thus $Z^{*}J \alpha z^{+} + J \alpha z^{-} = 0$, which implies $z^{-} = \alpha J Z^{*}J \alpha z^{+} = Jz^{+}$, whence $z^{+} + z^{-} \in W_{Z}$.
    This implies that $W_{Z}$ is maximal isotropic, whence it is Lagrangian, by \cref{lem:maximalMeansLagrangian}.
\end{proof}

\begin{lemma}\label{lem:PositiveGraphCondition}
    We have $\1 - Z^{*}Z > 0$ if and only if $W_{Z} \cap \alpha(W_{Z}) = \{ 0 \}$ and $W_Z \in \Pol^\omega(H)$. 
\end{lemma}
\begin{proof}
    First, assume that $\1 - Z^{*} Z > 0$.
    We have, for all $0 \neq x \in L^{+}$ that
    \begin{equation}\label{eq:ZtZ}
        x \neq Z^{*}Zx = \alpha Z \alpha Zx.
    \end{equation}
    Now, we have that $\alpha(W_{Z}) = \operatorname{graph}(\alpha Z \alpha)$, from which it follows that if $z \in W_{Z} \cap \alpha(W_{Z})$, then there exist $x \in L^{+}$ and $y \in L^{-}$ such that $(x,Zx) = z = (\alpha Z \alpha y, y)$.
    This implies that $x = \alpha Z \alpha Z x$, whence $x=0$ by \cref{eq:ZtZ}; and thus $W_{Z} \cap \alpha (W_{Z}) = \{ 0 \}$.
    
    Now, we compute, for arbitrary $0 \neq w = x+Zx \in W_{Z}$
    \begin{equation}\label{eq:PositivityRelation}
        \| P^{+} w \|^{2} - \| P^{-} w \|^{2} =  \|x \|^{2} - \|Zx \|^{2} = g(x,x) - g(x, Z^{*}Zx) = g(x, (\1 - Z^{*}Z)x)
    \end{equation}
    Now, we wish to show that $W_Z \in \Pol^\omega(H)$. 
    According to \cref{pr:CompatInequality}, it suffices to show that $\| P^{+}w \|^{2} > \|P^{-}w \|^{2}$ for all $0 \neq w \in W_{Z}$, this follows from the positivity of $\1 - Z^{*}Z$ through \cref{eq:PositivityRelation}.

Now, suppose that $W_{Z} \in \Pol^\omega(H)$ and 
set $g_{Z} = g_{W_{Z}}$ and $J_{Z} = J_{W_{Z}}$. \cref{eq:PositivityRelation} also implies (through \cref{pr:CompatInequality}) that, then $\1 - Z^{*}Z$ is positive.
\end{proof}

Set
\begin{equation}\label{eq:SiegelDiskDef}
    \mathfrak{D}(H) \defeq \{ Z \in \mc{B}(L^{+},L^{-}) \mid Z = \alpha Z^{*} \alpha, \text{ and } \1 - Z^{*}Z > 0 \}.
\end{equation}
This is called the Siegel disk, originating with C. L. Siegel \cite{siegel_symplectic_1943}.
    The Siegel disk in the setting of Hilbert spaces was defined by G. Segal \cite{segal_unitary_1981}. There, the extra condition that $Z$ is Hilbert-Schmidt was added, resulting in what we call the restricted Siegel disk below. 

The preceding discussion is now nicely summarized by the following result.
\begin{proposition}\label{prop:GrassmannianIsSiegelDisk}
    The maps
    \begin{align*}
        \Pol^{\omega}(H) & \rightarrow \mathfrak{D}(H), & \mathfrak{D}(H) & \rightarrow \Pol^{\omega}(H), \\
        W &\mapsto P^{-}(P^{+}|_{W})^{-1}, & Z & \mapsto \operatorname{graph}(Z),
    \end{align*}
    are well-defined, and each others' inverses.
\end{proposition}
\begin{proof}
    By \cref{pr:CompatInequality} we know that if $W \in \Pol^{\omega}(H)$, then $(P^{+}|_{W})^{-1}$ exists.
    By \cref{cor:GraphOperator} we then have that $W = \mathrm{graph}(P^{-}(P^{+}|_{W})^{-1})$.
    From \cref{lem:GraphLagrangianCondition,lem:PositiveGraphCondition} it then follows that $P^{-}(P^{+}|_{W})^{-1} \in \mathfrak{D}(H)$.
    Conversely, it follows from \cref{lem:GraphLagrangianCondition,lem:PositiveGraphCondition} that if $Z \in \mathfrak{D}(H)$, then $\operatorname{graph}(Z) \in \Pol^{\omega}(H)$.
    That these maps are each others inverses is a straightforward consequence of \cref{cor:GraphOperator}.
\end{proof}

\begin{remark}  \label{re:Siegel_upper_half}
    Another model of $\Pol^\omega(H)$ is given by the Siegel upper half space.  This is equivalent but based on different conventions. Rather than viewing positive symplectic polarizations as a variation of a fixed polarization $L^+ \oplus L^-$, instead one fixes a real Lagrangian decomposition. More precisely, we say that a Lagrangian subspace $L$ in $H_{\mathbb{C}}$ is real if it is the complexification of a Lagrangian subspace in $H$. Fix a decomposition $H_{\mathbb{C}} = L \oplus L'$ where $L$ and $L'$ are transverse real Lagrangian spaces.  It can be shown that any $W \in \Pol^\omega(H)$ is transverse to $L$ and $L'$, and can be uniquely expressed as the graph of some $Z:L \rightarrow L'$. This $Z$ satisfies $Z^T = Z$ and $\mathrm{Im}(Z)$ is positive definite. The set of such matrices is called the Siegel upper half space $\mathfrak{H}(H)$. These two conditions can be given meaning either through the use of a specific basis, or by $Z^T:= \alpha Z^* \alpha$ and $\mathrm{Im}(Z)  := (1/2i) (Z - \alpha \, Z \, \alpha)$.  For details, see \cite{berndt_introduction_2001,vaisman_symplectic_1987,siegel_symplectic_1943}. 
\end{remark} 

\begin{example}
   Here we refer to Example \ref{ex:polarization_compact_surface}.
   Let $\mathscr{R}$ be a compact Riemann surface, and $\mathcal{A}(\mathscr{R}) \oplus \overline{\mathcal{A}(\mathscr{R})}$ the decomposition of the space of complex harmonic one-forms on $\mathscr{R}$.
   This polarization, induced by the complex structure on the Riemann surface, is traditionally represented by an element of the Siegel upper half space as follows.
   Choose simple closed curves $a_1,\ldots,a_\mathfrak{g},b_1,\ldots,b_{\mathfrak{g}}$ which are generators of the first homology group of $\mathscr{R}$.
   These can be chosen so that their intersection numbers $\gamma_1 \cdot \gamma_2$ satisfy $a_k \cdot b_j = - b_j \cdot a_k = \delta_{kj}$, and are zero otherwise.  Intuitively, the intersection numbers count the  number of crossings with sign, where the sign of the crossing depends on the relative direction. For a precise definition of intersection number, and proof of existence of such a basis see \cite{farkas_riemann_1992,siegel_topics_1988}.  
   
   It can be shown that there are holomorphic one-forms $\beta_1,\ldots,\beta_{\mathfrak{g}}$ such that 
   \[  \int_{a_k} \beta_j = \delta_{jk}.  \]
   Then, 
   \[ Z_{kj} = \int_{b_k} \beta_j  \]
   is a matrix representing an element of the Siegel upper half space.
   The symmetry of this matrix is referred to as the Riemann bilinear relations, and $Z_{kj}$ is called the period matrix. 
   
 To see the relation to the operator specified in Remark \ref{re:Siegel_upper_half}, let $\eta_j$ be the unique basis of harmonic one-forms on $\mathscr{R}$ dual to the curves $a_1,\ldots,a_{\mathfrak{g}},b_1,\ldots,b_{\mathfrak{g}}$. Explicitly
 \[  \int_{a_k} \eta_j = \delta_{jk} \ \ \text{ and } 
   \int_{b_k} \eta_{j+\mathfrak{g}} = \delta_{jk}. \]
   Let $L$ be the span of $\{\eta_1,\ldots,\eta_{\mathfrak{g}} \}$ and $L'$ be the span of $\{\eta_{\mathfrak{g}+1},\ldots,\eta_{2\mathfrak{g}}\}$. Then it can be shown that
   \begin{equation}  \label{eq:holomorphic_oneforms_explicit}
    \beta_j = \eta_j + \sum_{k=1}^n Z_{jk} \eta_{k+\mathfrak{g}}, \ \ \  j = 1,\ldots,\mathfrak{g}.
   \end{equation}
   and 
   \[ \mathcal{A}(\mathscr{R}) = \mathrm{span} \{ \beta_1,\ldots,\beta_{\mathfrak{g}} \}. \]
   Thus if $Z_{jk}$ are the components of an operator $Z:L \rightarrow L'$ with respect to the bases above, then $Z$ is in the Siegel upper half plane and $\mathcal{A}(\mathscr{R})$ is the graph of $Z$. 
   See \cite{farkas_riemann_1992,siegel_topics_1988} for details. 
   
   Observe that an important role is played by the choice of homology basis (or marking).
\end{example}
\begin{exercise}  \label{exer:holomorphic_one_forms} 
    The Hodge theorem says that every cohomology class in $H^{1}_{\text{dR}}(\mathscr{R})$ is represented by a unique harmonic one-form.   Use the Hodge theorem 
    to show that the one-forms \cref{eq:holomorphic_oneforms_explicit} are indeed holomorphic and span $\mathcal{A}(\mathscr{R})$. 
\end{exercise}

Next, we return to the symplectic action on  $\Pol^\omega(H)$ and express it in terms of $\mathfrak{D}(H)$. 
Fix an element $u \in \Sp(H)$, we obtain a new subspace $W_{u}^{+} \defeq u(L^{+}) \subset H_{\C}$.
It is easy to see from the definition of $\Sp(H)$ that $W_{u}^{+}$ is again a Lagrangian, which moreover satisfies $W_{u}^{+} \cap \alpha W_{u}^{+} = \{ 0 \}$.
 
    The equation $\omega(ux,uy) = \omega(x,y)$ implies that 
    \begin{equation} \label{eq:J_commute_symplectic}
     -Ju^{*}Ju = \1
    \end{equation}
    Writing
    \begin{equation} \label{eq:block_form_symplectomorphism}
        u = \begin{pmatrix}
            a & \alpha  b \alpha \\
            b & \alpha a \alpha
        \end{pmatrix}:\begin{array}{c} L^+ \\ \oplus \\ L^- \end{array} \rightarrow \begin{array}{c} L^+ \\ \oplus \\ L^- \end{array} 
    \end{equation}
    in block form we obtain
    \begin{equation} \label{eq:symplectomorphism_identities}
        \1 = - \begin{pmatrix}
            i & 0 \\
            0 & -i
        \end{pmatrix}
        \begin{pmatrix}
            a^{*} & b^{*} \\
            \alpha b^{*} \alpha & \alpha a^{*} \alpha
        \end{pmatrix}
        \begin{pmatrix}
            i & 0 \\
            0 & -i
        \end{pmatrix}
        \begin{pmatrix}
            a & \alpha b \alpha \\
            b & \alpha a \alpha
        \end{pmatrix} = \begin{pmatrix}
            a^{*}a - b^{*}b & a^{*}\alpha b \alpha - b^{*} \alpha a \alpha \\
            - \alpha b^{*} \alpha a + \alpha a^{*} \alpha b & - \alpha b^{*} b \alpha + \alpha a^{*}a \alpha
        \end{pmatrix}.
    \end{equation} 

By a symplectomorphism, we mean a bounded bijection which preserves the symplectic form. Recall that the inverse is also bounded.
\begin{lemma} \label{le:top_left_invertible_symplecto}
    Let $u$ be a symplectomorphism in block form (\ref{eq:block_form_symplectomorphism}). Then $a$ is invertible and 
    \begin{equation} \label{eq:symp_inverse_formula}
      u^{-1} = \left( \begin{array}{cc} a^* & -b^* \\ -\alpha b^* \alpha & \alpha a^* \alpha \end{array} \right).  
    \end{equation}
\end{lemma}
\begin{proof}
    The expression for $u^{-1}$ follows from \cref{eq:J_commute_symplectic,eq:block_form_symplectomorphism}. If $av = 0$, since by \cref{eq:symplectomorphism_identities} $a^* a - b^* b = \1$  we see that for any $v \in L^+$
    \[  \| v \|^2 = \| av \|^2 - \| bv \|^2 =-\|bv\|^2  \]
    so $v=0$. Thus $a$ is injective. Applying this to $u^{-1}$ we see that $a^*$ is injective, so (denoting closure by $\mathrm{cl}$) 
    \[  \mathrm{cl} \, \mathrm{range}(a) = \mathrm{ker}(a^*)^\perp = L^+, \]
    i.e. $a$ has dense range. Since $\| v\|^2 \leq \| a v\|^2$, $a$ is bounded below, so it is invertible. 
\end{proof}
\begin{exercise} \label{exer:sympletic_inverse}
 Verify that $u^{-1}$ is given by equation (\ref{eq:symp_inverse_formula}), completing the proof of the lemma. 
\end{exercise}

\begin{corollary} \label{cor:SpActsOnSiegelDisk}
If $u \in \Sp(H)$ is given in block form by (\ref{eq:block_form_symplectomorphism}) and $Z \in \mathcal{D}(H)$, then $a + \alpha b \alpha Z$ is invertible. 
Under the bijection of Proposition \ref{prop:GrassmannianIsSiegelDisk}, $\Sp(H)$ acts transitively on $\mathcal{D}(H)$ via 
\[  Z  \mapsto (  {b}  + \alpha {a}  \alpha Z)(a + \alpha b \alpha Z)^{-1}.  \]
\end{corollary} 
\begin{proof}
    Fix $Z \in \mathcal{D}(H)$ and let $W_+$ be the graph of $Z$, so that
    \begin{align*}
     \Psi: L^+ & \rightarrow W_+  \\
     x & \mapsto x + Z x 
    \end{align*} 
    is an isomorphism. Let $W_+'=u W_+$ and observe that $\left. u \right|_{W_+}: W_+ \rightarrow W_+'$ is an isomorphism. Since  $\left. P_+' \right|_{W_+'}$ is also an isomorphism by Corollary \ref{cor:GraphOperator}, we see that 
    \[  a + \alpha b \alpha Z = P_+' \, u \, \Psi: L^+ \rightarrow L^+  \]
    is invertible. 

    We have already seen in Lemma \ref{lem:HomogeneousSets} that $\Sp(H)$ acts transitively on $\Pol^\omega(H)$, so the action induced on $\mathcal{D}(H)$ is transitive. Since $W_+'$ is the image of the graph of $Z$ under $u$, it is the image of 
    \[  \left( \begin{array}{c} a + \alpha b \alpha Z \\ b + \alpha a \alpha Z \end{array} \right) = u \left( \begin{array}{c} \1 \\ Z \end{array} \right):L^+ \rightarrow \begin{array}{c} L^+ \\ \oplus \\ L^- \end{array}.     \]
    This agrees with the image of 
    \[  \left( \begin{array}{c} \1 \\ (b + \alpha a \alpha Z)( a + \alpha b \alpha Z)^{-1} \end{array} \right)   \]
    which proves the claim. 
\end{proof}

\begin{example} \label{ex:smooth_example_bosonicH}
 Let $\bose_{\mathbb{C}}$ be as in Example \ref{ex:bosonExample}.  Assume that $\phi \in \mathrm{Diff}(\mathbb{S}^1)$. Modelling $\bose_{\mathbb{C}}$ by the space of functions $\dot{H}^{1/2}(\mathbb{S}^1)$ and using the expression 
 \[ \omega(f,g) = \int_{\mathbb{S}^1} f \, dg, \]
 for smooth $f,g$, we see by change of variables that 
 \begin{align*}
   \mathcal{C}_\phi: \dot{H}^{1/2}(\mathbb{S}^1) &\rightarrow \dot{H}^{1/2}(\mathbb{S}^1) \\f & \mapsto f \circ \phi - \frac{1}{2\pi} \int_{\mathbb{S}^1} f d\theta  
\end{align*}
preserves $\omega$ for smooth $f,g$. It can be shown that $\mathcal{C}_\phi$ is bounded for any diffeomorphism $\phi$. So the invariance of $\omega$ under $\mathcal{C}_\phi$ extends to all of $\bose$ by continuity of $\omega$.  Furthermore, $\phi^{-1} \in \mathrm{Diff}(\mathbb{S}^1)$ is a bounded inverse of $\mathcal{C}_\phi$ so we see that $\mathcal{C}_\phi$  
 is a symplectomorphism.
 
 Letting $L^\pm$ be as in Example \ref{ex:bosonExample}, we then have that 
 \[  W^\pm_\phi = \mathcal{C}_\phi L^+  \]
 defines an element of $\Pol^{\omega}(\bose)$. 
 The stabilizer of $L^+$ is precisely $\text{M\"ob}(\mathbb{S}^1)$, the set of M\"obius transformations preserving the circle. In summary, we obtain a well-defined injection 
 \begin{align*}
     \text{Diff}(\mathbb{S}^1)/\text{M\"ob}(\mathbb{S}^1) & \rightarrow \Pol^\omega(\bose) \\
     [\phi] & \mapsto W_\phi 
 \end{align*}
 where $[\phi]$ denotes the equivalence class of a representative $\phi \in \text{Diff}(\mathbb{S}^1)$.  
 
 It can be shown that the operator $Z \in \mathfrak{D}(\bose)$ associated to $W_\phi$ (see Corollary \ref{cor:GraphOperator}) is the Grunsky operator, which was introduced to complex function theory eighty years ago by H. Grunsky.
\end{example}
\begin{remark} \label{re:universal_Teich_space}
     Examples \ref{ex:smooth_example_bosonicH}, \ref{ex:real_boson_example}, and \ref{ex:bosonExample} originate with G. Segal \cite{segal_unitary_1981}. In the same paper, he introduced the concept of the infinite Siegel disk. 
    It was shown by Nag-Sullivan \cite{nag_teichmuller_1995} and Vodopy'anov \cite{vodopyanov_mappings_1989} that a homeomorphism of $\mathbb{S}^1$ is a symplectomorphism if and only if it is a quasisymmetry. Furthermore, the unitary subgroup  is the set of M\"obius transformations preserving the circle. 
    (Equivalently, by Exercise \ref{exer:stabilizer}, the unitary subgroup is the stabilizer of the polarization in Example \ref{ex:bosonExample}). 
    
    The definition of quasisymmetries is beyond the scope of this paper. We mention only that quasisymmetries modulo M\"obius transformations $\text{QS}(\mathbb{S}^1)/\text{M\"ob}(\mathbb{S}^1)$ is a model of the universal Teichm\"uller space. 
    Takhtajan and Teo showed that this gives a holomorphic embedding of the universal Teichm\"uller space $\text{QS}(\mathbb{S}^1)/\text{M\"ob}(\mathbb{S}^1)$ into the infinite Siegel disk. 
    The fact that the operator $Z$ is the Grunsky matrix was shown by Kirillov and Yuri'ev \cite{kirillov_representations_1988} in the smooth setting, and by Takhtajan and Teo \cite{takhtajan_weil-petersson_2006} in the quasisymmetric case. 
\end{remark}

 Next, we define the restricted Grassmannian and symplectic group. 
We first define a relation $\sim$ on $\Pol^{\omega}(H)$ by saying that $W_{1}^{+} \sim W_{2}^{+}$ if the restriction to $W_{1}^{+} \subset H_{\C}$ of the projection operator $H_{\C} = W_{2}^{+} \oplus \alpha (W_{2}^{+}) \rightarrow \alpha (W_{2}^{+})$ is Hilbert-Schmidt.
One might object that this definition is too vague, since it does not specify with respect to which inner product this operator is supposed to be Hilbert-Schmidt (reasonable options being $g$, $g_{W_{1}}$, and $g_{W_{2}}$).
The following result tells us that it does not matter.
\begin{lemma}
    If $H$ is a Hilbert space, equipped with two strong inner products, $g_{1}$ and $g_{2}$, then an operator $T:H \rightarrow H$ is Hilbert-Schmidt with respect to $g_{1}$ if and only if it is Hilbert-Schmidt with respect to $g_{2}$.
\end{lemma}
\begin{proof}
    Choose bases $\{ e_{i} \}$ and $\{ f_{i} \}$ which are orthonormal for $g_{1}$ and $g_{2}$ respectively.
    Let $A: H \rightarrow H$ be the complex-linear extension of $e_{i} \mapsto f_{i}$.
    The map $A$ is bounded linear, with bounded linear inverse by \cref{pr:StrongEquivalence}.
    Moreover, $A$ satisfies $g_{1}(v,w) = g_{2}(Av,Aw)$ for all $v,w \in H$.

    Denote by $\| - \|_{i}$ the Hilbert-Schmidt norm w.r.t.~$g_{i}$.
    Now, if $T$ is Hilbert-Schmidt w.r.t~$g_{1}$, then so is $A^{-1}TA$, and we compute
    \begin{align*}
        \| ATA^{-1} \|_{1}^{2} &= \sum_{i} g_{1}( A^{-1}TA e_{i}, A^{-1}TA e_{i}) \\
        &= \sum_{i} g_{2}(T f_{i}, T f_{i}) \\
        &= \|T\|_{2}^{2}. \qedhere
    \end{align*}
\end{proof}

\begin{lemma}\label{lem:BlockHilbertSchmidt}
    Let $u \in \Sp(H)$.
    Then $uL^{+} \sim L^{+}$ if and only if $b$ is Hilbert-Schmidt, where $b$ is given by the decomposition \cref{eq:block_form_symplectomorphism}.
\end{lemma}
\begin{proof}
    The projection $uL^{+} \rightarrow \alpha(L^{+})$ is given by $(ax,bx) \mapsto bx$.
    Because the operator $a$ is invertible (by \cref{le:top_left_invertible_symplecto}), we have that this operator is Hilbert-Schmidt if and only if $b$ is Hilbert-Schmidt.
\end{proof}

\begin{proposition}\label{prop:SymplecticEquivalence}
    The relation $\sim$ on $\Pol^{\omega}(H)$ is an equivalence relation.
\end{proposition}
\begin{proof}
    It is clear that $\sim$ is reflexive, because the projection operator $W^{+} \rightarrow \alpha(W^{+})$ is identically zero.
    
    Now, suppose that $W_{1}^{+} \sim W_{2}^{+}$.
    We may, without loss of generality, assume that $W_{2}^{+} = L^{+}$.
    There then exists an element $u \in \Sp(H)$ such that $uL^{+} = W_{2}^{+}$.
    
    Let $q:H_{\C} \rightarrow u(L^{-})$ be the projection with respect to the splitting $H_{\C} = u(L^{+}) \oplus u(L^{-})$.
    We determine the decomposition of $q$ with respect to the splitting $H_{\C} = L^{+} \oplus L^{-}$.
    Let $x \in H_{\C}$.
    Then, there exist $v^{\pm} \in u(L^{\pm})$ such that $x = v^{+} + v^{-}$; and we have $q(x) = v^{-}$.
    We apply $u^{-1}$ to obtain $u^{-1}x = u^{-1}v^{+} + u^{-1}v^{-}$.
    We observe that $u^{-1}v^{\pm} \in L^{\pm}$, thus projecting onto $L^{-}$, we obtain $P^{-}u^{-1}x = u^{-1}v^{-}$.
    Finally, we apply $u$ to obtain $uP^{-}u^{-1}(x) = v^{-} = q(x)$.
    We now compute
    \begin{equation*}
        uP^{-}u^{-1} = \begin{pmatrix}
            a & \alpha b \alpha \\
            b & \alpha a \alpha
        \end{pmatrix} \begin{pmatrix}
            0 & 0 \\
            0 & 1
        \end{pmatrix} \begin{pmatrix}
            a^{*} & -b^{*} \\
            - \alpha b^{*} \alpha & \alpha a^{*} \alpha 
        \end{pmatrix} = \begin{pmatrix}
            0 & \alpha b \alpha \\
            0 & \alpha a \alpha 
        \end{pmatrix} \begin{pmatrix}
            a^{*} & -b^{*} \\
            - \alpha b^{*} \alpha & \alpha a^{*} \alpha 
        \end{pmatrix} =
        \begin{pmatrix}
            - \alpha bb^{*} \alpha & \alpha b a^{*} \alpha \\
            - \alpha a b^{*} \alpha & \alpha aa^{*} \alpha
        \end{pmatrix}
    \end{equation*}
    The restriction of $q$ to $L^{+}$ is simply given by the first column of this matrix.
    We thus see that if $b$ is Hilbert-Schmidt, then $q$ is Hilbert-Schmidt.
    The assumption is that $uL^{+} \sim L^{+}$, which tells us that $b$ is Hilbert-Schmidt, through \cref{lem:BlockHilbertSchmidt}.
    It follows that $q$ is Hilbert-Schmidt, thus $L^{+} \sim u L^{+}$.
    
    Finally, suppose now that we have $W_{1}^{+} \sim W_{2}^{+}$ and $W_{2}^{+} \sim W_{3}^{+}$.
    Let $\iota_{1}^{+}:W_{1}^{+} \rightarrow H_{\C}$ and $\iota_{2}^{+}: W_{2}^{+} \rightarrow H_{\C}$ be the inclusions, and let $P_{i}^{\pm}: H_{\C} \rightarrow W_{i}^{\pm}$ be the projection with respect to the decompositions $H_{\C} = W_{i}^{+} \oplus W_{i}^{-}$ for $i=1,2,3$.
    We then have that the operators $P_{2}^{-}\iota_{1}^{+}$ and $P_{3}^{-}\iota_{2}^{+}$ are Hilbert-Schmidt.
    We have
    \begin{equation*}
        P_{3}^{-}\iota_{1}^{+} = P_{3}^{-}(\iota_{2}^{+}P_{2}^{+} + \iota_{2}^{-}P_{2}^{-})\iota_{1}^{+} = P_{3}^{-} \iota_{2}^{+} P_{2}^{+} \iota_{1}^{+} + P_{3}^{-} \iota_{2}^{-} P_{2}^{-} \iota_{1}^{+}.
    \end{equation*}
    It follows that $P_{3}^{-} \iota_{1}^{+}$ is Hilbert-Schmidt, and we are done.\textbf{}
\end{proof}

We introduce the \emph{restricted symplectic Lagrangian Grassmannian}
\begin{align*}
    \Pol_{2}^{\omega}(H) &\defeq \{ W^{+} \in \Pol^{\omega}(H) \mid W^{+} \sim L^{+} \} = \{ W^{+} \in \Pol^{\omega}(H) \mid L^{+} \rightarrow \alpha(W^{+}) \text{ is Hilbert-Schmidt} \},
\end{align*}
where $L^{+} \rightarrow \alpha(W^{+})$ is the restriction of the projection $H_{\C} = W^{+} \oplus \alpha(W^{+}) \rightarrow \alpha(W^{+})$ to $L^{+} \subset H_{\C}$.
Similarly, we may consider the \emph{restricted Siegel disk}
\begin{equation*}
    \mathfrak{D}_{2}(H) \defeq \mathfrak{D}(H) \cap \mc{B}_{2}(L^{+}, L^{-}),
\end{equation*}
where $\mc{B}_{2}(L^{+},L^{-})$ is the space of Hilbert-Schmidt operators from $L^{+}$ to $L^{-}$.

\begin{lemma} \label{le:LG_Siegel_disk_isomorphism_res}
    The isomorphism $\Pol^{\omega}(H) \rightarrow \mathfrak{D}(H)$ restricts to an isomorphism $\Pol_{2}^{\omega}(H) \rightarrow \mathfrak{D}_{2}(H)$.
\end{lemma}
\begin{proof}
    First, the image of $\Pol_{2}^{\omega}(H)$ in $\mathfrak{D}(H)$ under the map above lies in $\mathfrak{D}_{2}(H)$, because $P^{-}$ is Hilbert-Schmidt.

    To prove the converse, we observe that the orthogonal projection $p:\mathrm{graph}(Z) \rightarrow \alpha(L^{+})$ is given by $(x,Zx) \mapsto Zx$.
    If $\hat{Z}: L^{+} \rightarrow \mathrm{graph}(Z)$ is the invertible operator $x \mapsto (x,Zx)$ we thus have $Z = p \hat{Z}$.
    This implies that if $Z$ is Hilbert-Schmidt, then $p$ must be Hilbert-Schmidt. 
    In other words, $\mathrm{graph}(Z) \sim L^{+}$.
    The symmetry of $\sim$ then implies that $\mathrm{graph}(Z) \in \Pol_{2}^{\omega}(H)$.
\end{proof}

The \emph{restricted symplectic group} is
\begin{equation*}
	\Sp_{2}(H) \defeq \{ u \in \Sp(H) \, : \, u(L^+) \sim L^{+} \}. 
\end{equation*}

\begin{proposition} \label{pr:symp_res_characterization}
    Let $u \in \Sp(H)$. The following are equivalent. 
    \begin{enumerate}
        \item $u \in \Sp_{2}(H)$;
        \item If $u$ is given in block form by (\ref{eq:block_form_symplectomorphism}) then $b$ is Hilbert-Schmidt;
        \item $u^{\pullb}J - J \text{ is Hilbert-Schmidt}$.  
    \end{enumerate}
\end{proposition}
\begin{proof}
  By Lemma \ref{le:top_left_invertible_symplecto} $a$ is invertible. We have $u L^+$ is the image of $(a,\overline{b})^T L^+$, and thus is the graph of $\overline{b} a^{-1}$. Again since $a$ is invertible, $b$ is Hilbert-Schmidt if and only if $Z$ is Hilbert-Schmidt.  By  Lemma \ref{le:LG_Siegel_disk_isomorphism_res} conditions 1 and 2 are equivalent.  The equivalence of 2 and 3 follows from the computation
  \begin{equation} \label{eq:J_deformation}
    u^{-1} J u - J = 2i \left( \begin{array}{cc}   b^*b &  b^* \alpha a \alpha \\ - \alpha b^* \alpha a & - \alpha b^* b \alpha \end{array} \right). 
  \end{equation}
\end{proof}
\begin{exercise} \label{exer:complete_Hilbert_Schmidt_char}
    Show that \cref{eq:J_deformation} holds.
\end{exercise}

The actions of $\Sp(H)$ on the Siegel disk and Lagrangian Grassmannian restrict appropriately.
\begin{proposition} \label{pr:res_sym_preserves}
    The actions of $\Sp_{2}(H)$ preserves $\Pol_2^{\omega}(H)$ and $\mathfrak{D}_2(H)$. 
\end{proposition}
\begin{proof}
    That $\Sp_{2}(H)$ preserves $\Pol_{2}^{\omega}(H)$ is clear by definition.

    That $\Sp_{2}(H)$ preserves $\mathfrak{D}_{2}(H)$ follows from \cref{cor:SpActsOnSiegelDisk} together with part 2 of \cref{pr:symp_res_characterization}.
\end{proof} 

\begin{theorem} The Siegel disk $\mathfrak{D}(H)$ is an open subset of the Banach space of bounded linear operators on $H$ satisfying $\alpha Z \alpha = Z^*$, and in particular is a complex Banach manifold. The restricted Siegel disk $\mathcal{D}_2(H)$ is an open subset of the Hilbert space of Hilbert-Schmidt operators on $H$ satisfying $\alpha Z \alpha = Z^*$, and in particular is a complex Hilbert manifold.   
\end{theorem}
\begin{proof}
    Let $X$ denote the set of operators satisfying $\alpha Z \alpha = Z^*$. $X$ is a closed linear subspace of $\mc{B}(L^+,L^-)$, so it is itself a Banach space.  
    We have that $I- Z^* Z$ is positive definite if and only if $\| Z \|<1$. Thus $\mathcal{D}(H)$ is an open subset of $\mc{B}(L^+,L^-)$, and hence also an open subset of $X \cap \mc{B}(L^+,L^-)$. 

    The Hilbert-Schmidt norm controls the operator norm; that is, inclusion from the Hilbert-Schmidt operators into the bounded linear operators is a bounded map. Thus $\mathcal{D}_2(H)$ is an open subset of the space of Hilbert-Schmidt operators, and the remaining claims follow similarly.
\end{proof} 
By Propositions \ref{prop:GrassmannianIsSiegelDisk} and \ref{pr:res_sym_preserves} this gives the Grassmannian and restricted Grassmannian Banach and Hilbert manifold structures respectively.

\begin{example}  Referring to Example \ref{ex:smooth_example_bosonicH} and Remark \ref{re:universal_Teich_space}, it was shown by Segal that for smooth $\phi$ the operator $\mathcal{C}_\phi$ is an element of $\Sp_{2}(H)$. It was shown by Takhtajan and Teo that for a quasisymmetry $\phi$, $\mathcal{C}_\phi$ is in $\Sp_{2}(H)$ if and only if $\phi$ is what is known as a Weil-Petersson quasisymmetry. The definition is beyond the scope of the paper; we mention only that the Weil-Petersson quasisymmetries modulo the M\"obius transformations of the disk, $\text{QS}_{\text{WP}}(\mathbb{S}^1)/\text{M\"ob}(\mathbb{S}^1)$, is a model of the Weil-Petersson universal Teichm\"uller space.  
\end{example}

\section{The Grassmannian of polarizations associated to an inner product}\label{sec:fermions}

In this section, we fix an inner product, and describe the Lagrangian Grassmannian of orthogonal polarizations associated to that inner product. We show that it is a complex Banach manifold.
We also consider the ``restricted'' Grassmannian and orthogonal groups, and established that the the restricted Grassmannian is a Hilbert manifold.

We assume again that $H$ is equipped with a compatible triple $(g,J,\omega)$, and we write $H_{\C} = L^{+} \oplus L^{-}$ for the decomposition w.r.t.~$J$, i.e.~$L^{\pm} = \ker(J \mp i)$.
Now, choose an orthonormal basis $\{ e_{i} \}_{i \geqslant 1}$ for $L^{+}$.
For $i \leqslant 1$ we set $e_{i} = \alpha e_{-i}$, which yields an orthonormal basis $\{ e_{i} \}_{i \leqslant 1}$ for $L^{-} = \alpha(L^{+})$.
Nothing material will depend on these choices, but they will be convenient for the proofs to come.

Motivated by \cref{thm:FermionEquivalence} we introduce the \emph{restricted Lagrangian Grassmannian}
\begin{align*}
    \Pol_{2}^{g}(H) &\defeq \{ W^{+} \in \Pol^{g}(H) \mid L^{+} \rightarrow \alpha(W^{+}) \text{ is Hilbert-Schmidt} \},
\end{align*}
where $L^{+} \rightarrow \alpha(W^{+})$ is the restriction of the orthogonal projection $H_{\C} \rightarrow \alpha(W^{+})$ to $L^{+} \subset H_{\C}$.
We define a relation $\sim$ on $\Pol^{g}(H)$ by saying that $W_{1}^{+} \sim W_{2}^{+}$ if the orthogonal projection $W_{1}^{+} \rightarrow \alpha (W_{2}^{+})$ is Hilbert-Schmidt.
\begin{exercise}\label{ex:EquivalenceRelation}
   Prove that $\sim$ defines an equivalence relation on $\Pol^{g}(H)$. (This exercise is challenging.)
\end{exercise}
We also introduce the \emph{restricted orthogonal group}
\begin{equation*}
    \O_{2}(H) \defeq \{ u \in \O(H) \mid u(L^{+}) \in \Pol_{2}^{g}(H) \}.
\end{equation*}
Given an operator $u \in \O(H)$, let us write
\begin{equation*}
    u = \begin{pmatrix}
        a & b \\
        c & d
    \end{pmatrix},
\end{equation*}
with respect to the decomposition $H_{\C} = L^{+} \oplus L^{-}$.
\begin{lemma}
    We have that $u \in \O_{2}(H)$ if and only if $c$ is Hilbert-Schmidt, if and only if $b$ is Hilbert-Schmidt.
\end{lemma}
\begin{proof}
    First, we observe that $\alpha(u(L^{+})) = u(L^{-})$.
    The set $\{ ue_{i} \}_{i \leqslant 1} = \{ (be_{i},de_{i}) \}_{i \leqslant 1}$ is then an orthonormal basis for $u(L^{-})$.
    The orthogonal projection of $e_{j}$ onto $u(L^{-})$, for $j \geqslant 1$ is given by
    \begin{equation*}
        \sum_{i \leqslant 1} \langle e_{j}, be_{i} \rangle (be_{i},de_{i}).
    \end{equation*}
    It follows that the Hilbert-Schmidt norm of the orthogonal projection $L^{+} \rightarrow u(L^{-})$ is given by
    \begin{align*}
        \sum_{j \geqslant 1} \|\sum_{i \leqslant 1} \langle e_{j}, be_{i} \rangle (be_{i},de_{i}) \|^{2} &= \sum_{-i,j \geqslant 1} | \langle e_{j}, b e_{i} \rangle |^{2},
    \end{align*}
    which is simply the Hilbert-Schmidt norm of the operator $b: L^{-} \rightarrow L^{+}$.
    It follows that $u(L^{+}) \in \O_{2}(H)$ if and only if $b$ is Hilbert-Schmidt.
    That $b$ is Hilbert-Schmidt if and only if $c$ is Hilbert-Schmidt follows from the equation $\alpha c = b \alpha$.
\end{proof}
\begin{corollary}
    If $u \in \O_{2}(H)$, then $a$ and $d$ are Fredholm.
\end{corollary}
\begin{proof}
    It follows from the fact that $uu^{*} = \1$ that $aa^{*} + bb^{*} = \1$, and from $u^{*}u = \1$ that $a^{*}a + c^{*}c = \1$.
    Because $b$ and $c$ are Hilbert-Schmidt, we have that $bb^{*}$ and $c^{*}c$ are trace-class, thus in particular compact.
    It follows that $a$ is Fredholm (and $a^{*}$ is a ``parametrix'' for $a$).
\end{proof}
Now, suppose that $Z: L^{+} \rightarrow L^{-}$ is a linear operator.
The computation
\begin{equation*}
    \langle (v,Zv), \alpha (w, Zw) \rangle = \langle v, \alpha Zw \rangle + \langle Zv, \alpha w \rangle = \langle v, (\alpha Z + Z^{*} \alpha)w \rangle,
\end{equation*}
shows that $\gra(Z)$ is perpendicular to $\alpha (\gra(Z))$ if and only if $\alpha Z \alpha = -Z^{*}$.
\begin{exercise}\label{ex:PolarCondition}
    Suppose that $\alpha Z \alpha = -Z^{*}$, and show that if $(x,y) \in H_{\C}$ is perpendicular to both $\gra(Z)$ and $\alpha(\gra(Z))$, then $(x,y) = 0$.
\end{exercise}
It follows from \cref{ex:PolarCondition} that $\gra(Z) \in \Pol^{g}(H)$ if and only if $\alpha Z \alpha = - Z^{*}$.
Following \cite[Section 7.1]{pressley_loop_2003}, we equip $\Pol_{2}^{g}(H)$ with the structure of complex manifold.
To be precise, we shall equip $\Pol_{2}^{g}(H)$ with an atlas modeled on the complex Hilbert space
\begin{equation*}
    \mc{B}_{2}^{\alpha}(L^{+},L^{-}) = \{ Z \in \mc{B}_{2}(L^{+},L^{-}) \mid \alpha Z \alpha = - Z^{*} \},
\end{equation*}
such that the transition functions are holomorphic.
For $S$ a finite subset of the positive integers let $W_{S}$ be the closed linear span of
\begin{equation*}
    \{ e_{k}, e_{-l} \mid k \notin S, \, l \in S \}.
\end{equation*}
We see that $W_{S} \in \Pol_{2}^{g}(H)$.
We shall use the set of finite subsets of the positive integers as index set for our atlas.
\begin{lemma}\label{lem:FredholmZero}
    If $W^{+} \in \Pol_{2}^{g}(H)$, then the orthogonal projection $P:W^{+} \rightarrow L^{+}$ is a Fredholm operator, and $\mathrm{coker}(P) = \alpha \ker(P)$.
    In particular, the index of $P$ is zero.
\end{lemma}
\begin{proof}  
    Let $W^{+} \in \Pol_{2}^{g}(H)$, and write $P: W^{+} \rightarrow L^{+}$ for the orthogonal projection.
    There exists an operator $u \in \O_{2}(H)$ such that $u(L^{+}) = W^{+}$.
    As before, we write
    \begin{equation*}
        u = \begin{pmatrix}
            a & b \\
            c & d
        \end{pmatrix},
    \end{equation*}
    where $a: L^{+} \rightarrow L^{+}$ is Fredholm.
    Now, we claim that the map
    \begin{equation*}
        L^{+} \rightarrow W^{+}, x \mapsto u(x),
    \end{equation*}
    restricts to a bijection from $\ker(a)$ to $\ker(P)$.
    Indeed, $x \in \ker(a) \subset L^{+}$ if and only if $u(x) = (0, cx)$.
    Similarly $v \in \ker(P) \subset W^{+} = u(L^{+})$ if and only if $v = (0,cx)$.
    It thus follows that $\ker(P)$ is finite-dimensional, because $\ker(a)$ is (because $a$ is Fredholm).
    Now if $x \in L^{+}$ is arbitrary, then we have $u(x) = (ax, cx)$, and thus $Pu(x) = ax)$.
    it follows that the range of $P$ is given by the range of $a$, which is closed, and has finite-dimensional cokernel, because $a$ is Fredholm.
    
    Observe that the kernel of $P$ is given by $W^{+} \cap \alpha(L^{+})$.
    We claim that the cokernel of $P$ is $\alpha (W^{+}) \cap L^{+}$.
    In other words $\mathrm{Im}(P)^{\perp} = \alpha (W^{+}) \cap L^{+}$.
    Let $v \in \alpha(W^{+}) \cap L^{+}$ and $w \in W^{+}$, and denote by $P^{+}: H_{\C} \rightarrow L^{+}$ the orthogonal projection, so that $P^{+}|_{W^{+}} = P$.
    We then compute
    \begin{equation*}
        \langle v, Pw \rangle = \langle v, P^{+}w \rangle = \langle P^{+}v, w \rangle = \langle v,w \rangle = 0. 
    \end{equation*}
    So $v \in \mathrm{Im}(P)^\perp$.  Conversely, if $v \in \mathrm{Im}(P)^\perp$, then  $0=\langle v, Pw \rangle =  \langle v, w \rangle $ for all $w \in W^+$ so $v \in \alpha(W^+)$.
\end{proof}
\begin{corollary}\label{cor:EquivalentFredholmZero}
    If $W_{1}^{+} \sim W_{2}^{+}$, then the orthogonal projection $W_{1}^{+} \rightarrow W_{2}^{+}$ is a Fredholm operator of index zero.
\end{corollary}
\begin{proof}
    This follows from \cref{lem:FredholmZero} together with the transitivity of the relation $\sim$, see \cref{ex:EquivalenceRelation}.
\end{proof}
\begin{remark}
\cref{lem:FredholmZero} should be compared to \cref{pr:CompatInequality}. In that case, the projection operator is not just Fredholm, but even invertible.

Moreover, \cref{lem:FredholmZero} should be compared to e.g.~\cite[Prop.~6.2.4]{pressley_loop_2003}; there, any index can occur.
\end{remark}

We define
\begin{equation*}
    \mc{O}_{S} \defeq \{ \mathrm{graph}(Z) \in \Pol_{2}^{g}(H) \mid Z \in \mc{B}_{2}^{\alpha}(W_{S},W_{S}^{\perp}) \}.
\end{equation*}
The following result tells us that the sets $\mc{O}_{S}$ cover $\Pol_{2}^{g}(H)$, (see also \cite[Section 4]{borthwick_pfaffian_1992}, and c.f.~\cite[Proposition 7.1.6]{pressley_loop_2003}).
\begin{lemma}\label{lem:OpenCoverPolg}
    Any element of $\Pol_{2}^{g}(H)$ can be written as the graph of some operator $Z \in \HS^{\alpha}(W_{S},W_{S}^{\perp})$.
\end{lemma}
\begin{proof}
    Let $W^{+} \in \Pol_{2}^{g}(H)$ be arbitrary.
    We claim that there exists a finite subset $S$ of the positive integers, such that the projection operator $P_{S}:W^{+} \rightarrow W_{S}$ is an isomorphism.
    Once we have found such an $S$, it follows that $W^{+} = \mathrm{graph}(P^{\perp}_{S} P_{S}^{-1})$, where $P_{S}^{\perp}$ is the orthogonal projection of $W^{+}$ onto $W_{S}^{\perp}$, (c.f.~\cref{exer:VerifyGraphGivesSubspace}).
    The operator $P_{S}^{\perp}P_{S}^{-1}$ is Hilbert-Schmidt, because $P_{S}^{\perp}$ is.

    Now, let $E_{\emptyset} \subset W^{+}$ be the kernel of the projection operator $P: W^{+} \rightarrow L^{+} = W_{\emptyset}$.
    Let $d < \infty$ be its dimension.
    If $d=0$, it follows that $P$ is an isomorphism, because it is Fredholm of index $0$ by \cref{cor:EquivalentFredholmZero}.
    So, assume that $d > 0$, and let $0 \neq v \in E_{\emptyset}$ be a unit vector.
    Choose a positive integer $l$ such that $\langle v, e_{-l} \rangle \neq 0$.
    We now consider the orthogonal projection $P_{\{ l \}}: W^{+} \rightarrow W_{\{l\}}$; and we let $E_{\{ l \}}$ be its kernel.
    We claim that $E_{\{ l \} }$ is a strict subset of $E_{\emptyset}$, and thus that the dimension of $E_{\{ l \} }$ is strictly smaller than the dimension of $E_{\emptyset}$.
    First, observe that $v \notin E_{\{ l \} }$, while $v \in E_{\emptyset}$.
    It thus remains to be shown that $E_{ \{l \} }$ is a subset of $E_{\emptyset}$ in the first place.
    Suppose that $x \in E_{\{ l \} }$, and write $x = \sum_{k} \lambda_{k} e_{k}$.
    The condition that $x \in E_{\{ l \} }$ then implies that
    \begin{equation*}
        \lambda_{k} = 0, \quad k \geqslant 1, k \neq l.
    \end{equation*}
    Now, we have that $v = \sum_{k} \mu_{k} e_{k}$, where $\mu_{k} = 0$ for $k \geqslant 1$, and $\mu_{-l} = \langle v, e_{-l} \rangle \neq 0$.
    We have
    \begin{equation*}
        \alpha(x) = \sum_{k} \overline{\lambda_{-k}} e_{k}.
    \end{equation*}
    We then use that $\alpha(x) \in W^{-} = (W^{+})^{\perp}$ and compute
    \begin{equation*}
        0 = \langle \alpha(x), v \rangle = \sum_{k} \overline{\lambda_{-k}} \mu_{k} = \overline{\lambda_{l}} \mu_{-l}.
    \end{equation*}
    The condition that $\mu_{-l} \neq 0$ then implies that $\lambda_{l} = 0$, whence $x \in E_{\emptyset}$.
    
    By induction (using \cref{cor:EquivalentFredholmZero}) on this dimension, we construct a finite subset of the positive integers $S$, such that $P_{S}:W^{+} \rightarrow W_{S}$ is injective.
    Because this operator is Fredholm of index zero (again by \cref{cor:EquivalentFredholmZero}), it follows that it is an isomorphism, which completes the proof.
\end{proof}

Fix now finite subsets of the positive integers $S_{1}$ and $S_{2}$.
We determine the ``transition functions'' corresponding to the charts $\mc{B}_{2}(W_{S_{1}},W_{S_{1}}^{\perp}) \leftarrow \mc{O}_{S_{1}} \cap \mc{O}_{S_{2}} \rightarrow \mc{B}_{2}(W_{S_{2}},W_{S_{2}}^{\perp})$, following \cite[Proposition 7.1.2]{pressley_loop_2003}.
Let $a,b,c,d$ be operators such that
\begin{equation*}
    \begin{pmatrix}
        a & b \\
        c & d
    \end{pmatrix}: \begin{pmatrix}
        W_{S_{1}} \\
        W_{S_{1}}^{\perp}
    \end{pmatrix} \rightarrow \begin{pmatrix}
        W_{S_{2}} \\
        W_{S_{2}}^{\perp}
    \end{pmatrix}
\end{equation*}
is the identity operator.
A straightforward verification shows that $a$ and $d$ are Fredholm, and that $b$ and $c$ are Hilbert-Schmidt (indeed, even finite-rank).

\begin{proposition} \label{pr:transition_functions_HS_fermion}
    The image of $\mc{O}_{S_{1}} \cap \mc{O}_{S_{2}}$ in $\mc{B}_{2}^{\alpha}(W_{S_{1}},W_{S_{1}}^{\perp})$ consists of those operators $Z_{1}$ such that $a+bZ_{1}$ has a bounded inverse.
    Moreover, the transition function is given by $Z_{1} \mapsto (c+dZ_{1})(a+bZ_{1})^{-1}$.
\end{proposition}
\begin{proof}
    Let $W^{+} \in \mc{O}_{S_{1}} \cap \mc{O}_{S_{2}}$, we then have $\mathrm{graph}(Z_{1}) = W^{+} = \mathrm{graph}(Z_{2})$, for some $Z_{i} \in \mc{B}_{2}^{\alpha}(W_{S_{i}},W_{S_{i}}^{\perp})$.
    The orthogonal projection $P_{i}:W^{+} \rightarrow W_{S_{i}}$ is invertible, with inverse $x \mapsto (x,Z_{i}x)$.
    This implies that the composition $P_{2}P_{1}^{-1}: W_{S_{1}} \rightarrow W_{S_{2}}$ is invertible.
    The calculation
    \begin{align*}
        P_{2}P_{1}^{-1}x = P_{2}(x, Z_{1}x) = ax + bZ_{1}x
    \end{align*}
    shows that $P_{2}P_{1}^{-1} = a + bZ_{1}$.
    
    On the other hand, if $Z_{1} \in \mc{B}_{2}^{\alpha}(W_{S_{1}},W_{S_{1}}^{\perp})$ is such that $a+bZ_{1}$ has a bounded inverse, then we consider the operator $(c+dZ_{1})(a+bZ_{1})^{-1}:W_{S_{2}} \rightarrow W_{S_{2}}^{\perp}$.
    This operator is Hilbert-Schmidt, because both $c$ and $Z_{1}$ are Hilbert-Schmidt.
    Now, let $y \in W_{S_{2}}$ be arbitrary.
    Set $x = (a+bZ_{1})^{-1}y \in W_{S_{1}}$.
    We then compute
    \begin{equation*}
        \begin{pmatrix}
            y \\
            (c+dZ_{1})(a+bZ_{1})^{-1}y
        \end{pmatrix} = \begin{pmatrix}
            (a+bZ_{1})x \\
            (c+dZ_{1})x
        \end{pmatrix} = \begin{pmatrix}
            a & b \\
            c & d
        \end{pmatrix} \begin{pmatrix}
            x \\
            Z_{1}x
        \end{pmatrix}.
    \end{equation*}
    This proves that the graph of $Z_{1}$ is equal to the graph of $(c+dZ_{1})(a+bZ_{1})^{-1}$.
\end{proof}

To prove that the transition functions form a complex atlas, we need an elementary lemma. 
\begin{lemma} \label{le:control_inverse}
    Let $A$ and $A'$ be bounded operators on a Banach space, and assume that $A$ is invertible.
    If $\| A - A' \|<1/ (2 \| A^{-1} \|)$ then $A'$ is invertible and 
    \[ \|(A')^{-1} \| \leq 2 \| A^{-1} \|. \] 
\end{lemma}
\begin{proof}
    Since $\| (A'-A) A^{-1} \| <1/2$, the Neumann series provides an inverse for $I + (A'-A)A^{-1}$ and   
    \[  \left\| \left[ I + (A'- A)A^{-1} \right]^{-1} \right\| \leq \frac{1}{1-\| (A'-A) A^{-1} \|} <2. \]
    Thus 
    \[ (A')^{-1} = A^{-1} \left( I + (A'- A)A^{-1} \right)^{-1}  \]
    and the claim follows directly. 
\end{proof}

\begin{theorem} The transition functions in Proposition \ref{pr:transition_functions_HS_fermion} are holomorphic.  
\end{theorem}
\begin{proof}
    Let $\Psi:Z \mapsto (c+dZ)(a+bZ)^{-1}$ be the transition map  in the space of Hilbert-Schmidt operators. It is enough to show that $\Psi$ is G\^ateaux differentiable and locally bounded. Denote the Hilbert-Schmidt norm by $\| \cdot \|_{HS}$ and the usual operator norm by $\| \cdot \|$. Two facts we will use are that $\| AB \|_{HS} \leq \| A \| \| B \|_{HS}$ and $\| A \| \leq \| A \|_{HS}$. 
    
    It is easily seen that the map is G\^ateaux holomorphic with G\^ateaux derivative 
    \[  D \Psi(Z;W) = \left[ dW + (c+d Z) (a+bZ)^{-1} bW \right](a+bZ)^{-1}. \] 
    To see that $\Psi$ is locally bounded, fix $Z$ and assume that 
    \[ \| Z' - Z \|<  \frac{1}{\| b\| \| (a + bZ)^{-1} \|}.  \]
    By Lemma \ref{le:control_inverse} with $A = a + bZ$, $A'=a+bZ'$ we then have
    \[  \left\| \left(a + b Z' \right)^{-1} \right\| \leq \frac{1}{2 \| (a + b Z)^{-1} \|}.   \]
    Thus 
    \[ \left\| (c + dZ')\left(a+bZ'\right)^{-1} \right\|_{HS} \leq \left( \| c \|_{HS} + \| d \| \| Z' \|_{HS} \right) \frac{1}{2 \| (a + b Z)^{-1} \|}   \]
    which proves the claim.
\end{proof}

\begin{remark}
The Grassmannian $\Pol_{2}^{g}(H)$ is a rich geometric object.
For example, it carries a ``Pfaffian line bundle'' \cite{borthwick_pfaffian_1992}.
\end{remark}

\section{Motivation for the restricted Grassmannian}\label{sec:HilbertSchmidt}

In this section, we give a sketch of the representation-theoretic and physical motivation for the the restricted spaces. 

Let $H$ be a real Hilbert space, with inner product $g$.
From this data, one can construct a \cstar-algebra, called the \emph{Clifford \cstar-algebra}.
We now give an overview of this construction, together with some aspects of the corresponding representation theory.
A full account of the theory is given \cite{plymen_spinors_1994} and \cite{jorgensen_bogoliubov_1987}.

The complex Clifford algebra of $H$ is the complex $*$-algebra generated by elements of $H_{\C}$, subject to the condition that $vw+wv = g(v,\alpha w)$ and $v^{*} = \alpha(v)$, for all $v,w \in H_{\C}$.
This algebra can be completed to a \cstar-algebra, which we denote by $\Cl(H)$.
For $W^{+} \in \Pol^{g}(H)$ we define $F_{W^{+}}$ to be the Hilbert completion of the exterior algebra of $W^{+}$:
\begin{equation*}
    F_{W^{+}} \defeq \overline{\bigoplus_{n \geqslant 0} \wedge^{n} W^{+}}.
\end{equation*}
We define a map $\pi_{W^{+}}: H_{\C} \rightarrow \mc{B}(F_{W^{+}})$ by setting
\begin{equation*}
    \pi_{W^{+}}(x) w_{1} \wedge ... \wedge w_{n} = x \wedge w_{1} \wedge ... \wedge w_{n}
\end{equation*}
for $x,w_{1},...,w_{n} \in W^{+}$, and $\pi_{W^{+}}(y) = \pi_{W^{+}}(\alpha(y))^{*}$ for $y \in W^{-}$.
The map $\pi_{W^{+}}$ defined in this way admits a unique extension to a $*$-homomorphism $\pi_{W^{+}}: \Cl(H) \rightarrow \mc{B}(F_{W^{+}})$; this is the Fock representation of $\Cl(H)$ corresponding to $W^{+}$, (\cite[Chapter 2]{plymen_spinors_1994}).

It is then natural to attempt to compare the Fock representations corresponding to two elements $W_{1}^{+},W_{2}^{+} \in \Pol^{g}(H)$.
The following result tells us when two such representations should be viewed as equivalent.
\begin{theorem}\label{thm:FermionEquivalence}
    There exists a unitary operator $u: F_{W_{1}^{+}} \rightarrow F_{W_{2}^{+}}$ with the property that
    \begin{equation*}
        \pi_{W_{1}^{+}}(a) = u^{*}\pi_{W_{2}^{+}}(a)u
    \end{equation*}
    for all $a \in \Cl(H)$ if and only if the orthogonal projection $H \rightarrow \alpha(W_{2}^{+})$ restricts to a Hilbert-Schmidt operator $W_{1}^{+} \rightarrow \alpha(W_{2}^{+})$
\end{theorem}
A proof of this classical result, together with references to the expansive literature on the subject, can be found in \cite[Chapter 3]{plymen_spinors_1994}.

If $H$ is a real Hilbert space with symplectic form $\omega$, then one defines the Heisenberg (Lie) algebra to be $H \times \R$, equipped with the bracket
\begin{equation*}
    [(v,t),(w,s)] = (0, \omega(v,w)).
\end{equation*}
Given a positive symplectic polarization $H_{\C} = W^{+} \oplus W^{-}$, one obtains a representation of the Heisenberg algebra on the Hilbert completion of the symmetric algebra of $W^{+}$, similar to the construction above.
However, this representation is by \emph{unbounded} operators. In spite of this, the situation is entirely analogous:
If $H$ is finite-dimensional, then all such representations are unitarily equivalent; this is the Stone-von Neumann theorem.
The equivalence problem in the infinite-dimensional case was settled by D.~Shale \cite{shale_linear_1962}.
Indeed, the Stone-von Neumann theorem is superseded by the result that two positive symplectic polarizations $W_{1}^{\pm}$ and $W_{2}^{\pm}$ give unitarily equivalent representations, if and only if they are in the same Hilbert-Schmidt class.
This statement is not quite precise, because we have not explained what a representation by unbounded operators is.
There are several ways to deal with this issue, but this would take us too far afield. 
The reader can find the subject treated in for example \cite{lion_weil_1980,ottesen_infinite_1995,habermann_introduction_2006}.

\section{Sewing and \texorpdfstring{$\mathrm{Diff}(\mathbb{S}^1)$}{Diff(S1)}} \label{se:sewing}

In this section, we return to Example \ref{ex:smooth_example_bosonicH}, and give the Grassmannian of polarizations a geometric interpretation in terms of sewing.  
This sewing interpretation arises in conformal field theory, see \cite{bleuler_definition_1988,huang_two-dimensional_1995}. This section can be treated as an extended example. This example appears in many contexts. 

We first give a summary. Let $\overline{\mathbb{C}}$ denote the Riemann sphere, 
\[  \mathbb{D}_+ = \{z \in \mathbb{C} \,: \, |z|<1 \}, \ \ \ \text{and} \ \ \ \mathbb{D}_-= \{ z \in \overline{\mathbb{C}} \, :\, |z|>1 \} \cup \{ \infty \}.  \]
Every diffeomorphism $\phi \in \mathrm{Diff}(\mathbb{S}^1)$ gives rise to a conformal map $f$ from the unit disk into the sphere, obtained by sewing $\disk_-$ to $\disk_+$ using $\phi$ to identify points on their boundaries.  The image of $f$ is bounded by a smooth Jordan curve representing the seam generated by the choice of $\phi$. If $\phi$ is the identity map, the seam is just the unit circle in the Riemann sphere. 

Now let $\Sigma$ be the complement of the closure of the image of $f$. Recall that $\phi$ induces a composition operator $\mathcal{C}_\phi$ on $\bose$, which is a bounded symplectomorphism. and that $W_\phi = \mathcal{C}_\phi W_+ \in \Pol^\omega(\bose)$. We then have that $W_\phi$ can be interpreted as pull-back by $f$ of the set of boundary values of holomorphic functions on $\Sigma$.  

To fill in this summary, we need a result known as conformal welding. It originated in quasiconformal Teichm\"uller theory in the 1960s (see \cite{lehto_univalent_1987} and references therein), and independently in other contexts including conformal field theory. We give the description in the smooth case because it allows a simpler presentation in terms of more well-known theorems. 

The conformal welding theorem (in the smooth case) says the following. 
\begin{theorem}[smooth conformal welding] Let $\phi \in \mathrm{Diff}(\mathbb{S}^1)$. There are holomorphic one-to-one functions $f:\mathbb{D}_+ \rightarrow \overline{\mathbb{C}}$, $g: \mathbb{D}_- \rightarrow \overline{\mathbb{C}}$, which extend smoothly and bijectively to $\mathbb{S}^1$, such that $f(\mathbb{S}^1)=g(\mathbb{S}^1)$ and $\phi = \left. g^{-1} \circ f \right|_{\mathbb{S}^1}$. 

These are unique up to post-composition by a M\"obius transformation; that is, any other such pair of maps is given by $T \circ f$, $T \circ g$. 
\end{theorem}
\begin{proof} 
We give a sketch of a proof. Given $\phi \in \mathrm{Diff}(\mathbb{S}^1)$, we treat it as a parametrization of the boundary of the Riemann surface $\mathbb{D}_-$. Sew on $\disk_+$ by identifying points on $\partial \disk_+$ and $\partial \disk_-$ under $\phi$. One then obtains a topological sphere $S = \mathrm{cl} \,\disk_+ \sqcup \mathrm{cl} \, \disk_-/\sim$, where $p \in \partial \disk_+$ is equivalent to $q \in \partial \disk_-$ if and only if $q = \phi(p)$. $S$ can be given a unique complex manifold structure compatible with that on $\disk_+$ and $\disk_-$. By the uniformization theorem, there is a biholomorphism $\Psi:S \rightarrow \sphere$. Set 
\[  f= \left. \Psi \right|_{\disk_+}, \ \ \ g = \left. \Psi \right|_{\disk_-}.  \]
It follows from continuity of $\Psi$ together with the definition of the equivalence relation $\sim$ that $\phi = g^{-1} \circ f$. 

    Uniqueness can be obtained from the fact that any biholomorphism of the Riemann sphere is a M\"obius transformation. 
\end{proof}
\begin{remark}
    In fact, the original version in quasiconformal Teichm\"uller theory was valid more generally for quasisymmetries of $\mathbb{S}^1$ (see Remark \ref{re:universal_Teich_space}). This is the theorem usually referred to as the conformal welding theorem.
\end{remark}

The proof of the smooth conformal welding theorem shows that in general, given a disk $\disk_-$ whose boundary is equipped with a parameterization $\phi \in \mathrm{Diff}(\mathbb{S}^1)$ one can sew on a disk $\disk_+$. This resulting Riemann surface is the Riemann sphere, but the seam is now a smooth Jordan curve $\Gamma$. Let $\Sigma = \Psi(\disk_-)$ denote the copy of $\disk_-$ in $\sphere$; the parameterization $\phi$ now is represented equivalently by the boundary values of $f$, which is a smooth function taking $\mathbb{S}^1$ to $\Gamma=f(\mathbb{S}^1)=g(\mathbb{S}^1)$.  

Without loss of generality, assume that $\infty \in \Sigma$. Let 
\[  \mathcal{D}_\infty(\Sigma) = \left\{ h:\Sigma \rightarrow \mathbb{C} \text{ holomorphic } \,:\, \iint_{\Sigma} |h'|^2 <\infty, \ h(\infty)=0 \right\}   \]
be the Dirichlet space. Given $h \in \mathcal{D}_\infty(\Sigma)$, if we assume that $h$ has a smooth extension to the boundary $\partial \Sigma$, then $h \circ f$ is a smooth function on $\mathbb{S}^1$, and in particular is in $\bose_{\mathbb{C}}$. In fact, it can be shown that $h$ extends to the boundary and $h \circ f$ makes sense for any $h \in \mathcal{D}_\infty(\Sigma)$, and furthermore $h \circ f -h(f(0))$ is in $\bose_{\mathbb{C}}$ \cite{schippers_analysis_2022}.
(In fact this holds for any quasisymmetry $\phi$, cf Remark \ref{re:universal_Teich_space}.)  We then have 
\begin{theorem} For $\phi \in \mathrm{Diff}(\mathbb{S}^1)$, if $W_\phi = \mathcal{C}_\phi L^+$, then 
\begin{equation*}
    W_\phi = f^\pullb \mathcal{D}_\infty(\Sigma)  : = \{ h \circ f -h(f(0)) \,:\, h \in \mathcal{D}_\infty(\Sigma) \}.
\end{equation*}
\end{theorem}

More generally, in his sketch of a definition of conformal field theory, Segal \cite{bleuler_definition_1988} considered the category whose objects are Riemann surfaces of genus $g$ with $n$ closed, boundary curves circles endowed with boundary parametrizations $(\phi_1,\ldots,\phi_n)$. After sewing on copies of the disk, one obtains a compact surface $\mathscr{R}$, holomorphic maps $f=(f_1,\ldots,f_n)$ representing the parametrizations, and $\Sigma$ can be identified with the complement of the closures of the images of $f_1(\disk_+),\ldots,f_n(\disk_+)$. The sets $f^*\mathcal{D}(\Sigma)$ play a role in the construction; the boundary parameterizations are a means to obtain Fourier series from the boundary values of elements of $\mathcal{D}(\Sigma)$, and the induced polarizations represent the positive and negative Fourier modes.

The moduli space $\widetilde{\mathcal{M}}(g,n)$ space of surfaces with boundary is as follows.
Two elements $(\Sigma,\phi_1,\ldots,\phi_n)$, $(\Sigma',\phi_1',\ldots,\phi_n')$ are equivalent if there is a conformal map $F:\Sigma \rightarrow \Sigma'$ such that $\phi_k' = F \circ \phi_k$ for $k=1,\ldots,n$. 

If the parameterizations are taken to be smooth in the definition of $\widetilde{M}(g,n)$, then since all type $(0,1)$ surfaces are equivalent to $\disk_-$ we can identify 
\[  \widetilde{M}(0,1) \cong \mathrm{Diff}(\mathbb{S}^1) /\text{M\"ob}(\mathbb{S}^1).    \] 

 The moduli space $\widetilde{\mathcal{M}}(0,1)$, as well as the universal Teichm\"uller space can be ebmedded in the Grassmannian of polarizations $\Pol^\omega(\bose)$. 
To see this, recall from Remark \ref{re:universal_Teich_space} that the stabilizer of $L^+$ in $\mathrm{Diff}(\mathbb{S}^1)$ is the set $\text{M\"ob}(\mathbb{S}^1)$ of M\"obius transformations preserving $\mathbb{S}^1$. This gives the following.
\begin{corollary}
 We have a well-defined injective map 
 \begin{align*}
    \widetilde{M}(0,1) &  \rightarrow \Pol^{\omega}(\bose)\\
    [(\disk_-,\phi)] & \mapsto W_\phi.
 \end{align*}
\end{corollary} 

The universal Teichm\"uller space $T(0,1)$ is a moduli space containing the Teichm\"uller spaces of all Riemann surfaces covered by the disk.
It is classically known to be modelled by $\text{QS}(\mathbb{S}^1)/\text{M\"ob}(\mathbb{S}^1)$. The embedding in Remark \ref{re:universal_Teich_space} is as follows. 
\begin{theorem} 
    There is a well-defined injective map \cite{nag_teichmuller_1995} from $T(0,1)=\text{QS}(\mathbb{S}^1)/\text{M\"ob}(\mathbb{S}^1)$ to $\Pol^\omega(\bose)$. 
 \begin{align*}
    T(0,1) &  \rightarrow \Pol^{\omega}(\bose)\\
    [\phi] & \mapsto W_\phi.
 \end{align*}
 This map is holomorphic with respect to the classical Banach manifold structures on $T(0,1)$ and $\mathcal{D}(\bose)$ \cite[Theorem B1]{takhtajan_weil-petersson_2006}.
\end{theorem}
\begin{remark} 
  If one uses quasisymmetric parametrizations in the definition of Segal's moduli space $\widetilde{M}(0,1)$, one then sees that it can be identified naturally with the universal Teichm\"uller space $T(0,1)$.  
    The desirability of this extension to quasisymmetries is strongly motivated by Nag-Sullivan/Vodop'yanov's result that the quasisymmetries produce precisely the composition operators which are bounded symplectomorphisms (see Remark \ref{re:universal_Teich_space}). It is remarkable that the link between these two spaces - a geometric fact - is established by an analytic condition. This analytic condition in turn is motivated by an algebraic requirement; namely, the requirement that the representation of quasisymmetries be bounded. Interestingly, this analytic condition was in place for independent reasons in Teichm\"uller theory decades before the result of Nag-Sullivan/Vodopy'anov.  
\end{remark}
\begin{remark}
The association between the Teichm\"uller space and the moduli space of Segal holds for arbitrary surfaces of genus $g$ with $n$ boundary curves.  In general, the Segal moduli space $\widetilde{M}(g,n)$, if extended to allow quasisymmetric parametrizations, is a quotient of the Teichm\"uller space $T(g,n)$ by a finite-dimensional modular group  \cite{radnell_quasisymmetric_2006}. (This modular group is trivial in the case $(g,n)=(0,1)$).  
\end{remark}

\begin{remark}
The extension of the symplectic action of $\mathrm{Diff}(\mathbb{S}^1)$ to quasisymmetries (resulting in the embedding of universal Teichm\"uller space), was recognized by Nag and Sullivan \cite{nag_teichmuller_1995}. The problem of determining which $\phi$ result generate elements of the restricted Grassmannian has roots in work of Nag and Verjovsky \cite{nag_rm_1990}, who considered the convergence of a natural generalization of the classical Weil-Petersson pairing on the tangent space to the universal Teichm\"uller space. This eventually gave birth to the Weil-Petersson universal Teichm\"uller space, now widely studied.  Takhtajan and Teo  \cite{takhtajan_weil-petersson_2006} showed that the embeddings of the universal Teichm\"uller space and Weil-Petersson Teichm\"uller space into $\Pol^\omega(\bose)$ and $\Pol_2^\omega(\bose)$ are holomorphic. The interpretation in terms of pulling back boundary values of functions in the Dirichlet space requires some justification \cite{schippers_analysis_2022}.  
For the induced representation on symmetric Fock space, see \cite{segal_unitary_1981} in the smooth case. This was extended to the Weil-Petersson class universal Teichm\"uller space by A. Serge'ev \cite{sergeev_lectures_2014}. A quantization procedure on the classical universal Teichm\"uller space is also described there and in references therein. 
\end{remark}

\section{Solutions} \label{se:solutions}

 \noindent {\bf{Exercise} \ref{exer:one_two_then_J_preserves}}. By definition, we have
 \[  g(Jv,Jw)= \omega(Jv,J^2w) = \omega(w,Jv) =g(v,w) \]
 and similarly using Proposition \ref{pr:three_properties_compatible} part 2, 
 \[  \omega(Jv,Jw) = g(-v,Jw) = -\omega(v,J^2w) = \omega(v,w).  \] 

\medskip \noindent  
   {\bf{Exercise} \ref{exer:two_of_three_gives_compat}}.
    \begin{enumerate}[label=\alph*)]
        \item A symplectic form that makes the triple $(g,J,\omega)$ compatible, if it exists, must be given by the equation $\omega(v,w) = g(Jv,w)$.
        It is clear that the induced map $\ph_{\omega}$ is  invertible, because $\ph_{\omega} = \ph_{g} J$.
        One sees that $\omega$ is anti-symmetric if and only if $J$ is skew-adjoint with respect to $g$ by the computation $\omega(v,w) = g(Jv,w) = -g(Jw,v) = \omega(w,v)$.   
        \item In this case, the equation $J = \ph_{g}^{-1} \ph_{\omega}$ defines a complex structure.
        \item We compute $\omega(Jv,w) = -\omega(w,Jv) = -g(w,v) = -g(v,w) = -\omega(v,Jw)$.
    \end{enumerate}

\medskip \noindent {\bf{Exercise} \ref{exer:example_not_inner}}. 
    If $(g,J,\omega)$ is a compatible triple, then replacing $J$ by $-J$ will give an example as required.


\medskip \noindent {\bf{Exercise} \ref{exer:hodge_well_defined_symplectic}}.
    Let $\delta$ be closed and $d\phi$ be exact.  By Stokes' theorem, since $d\delta =0$ 
    \[ \omega(\delta,d\phi) = \frac{1}{\pi} \iint_{\mathscr{R}} \delta \wedge d\phi = - \iint_{\mathscr{R}} d (\phi \delta) =0    \]
    where the final equality is because $\mathscr{R}$ has no boundary. Thus given any closed one-forms $\beta,\gamma$, with harmonic one-form representatives $\hat{\beta},\hat{\gamma}$ such that
    \[ \beta = \hat{\beta} + d\phi, \ \ \gamma = \hat{\gamma}+ d\psi, \]
    we have $\omega(\hat{\beta},d\psi)=\omega(d\phi,\hat{\gamma})=\omega(d\phi,d\psi)=0$ which proves the claim. 

\medskip \noindent {{\bf Exercise  \ref{exer:smooth_symplectic_Hb}}}. 
 One can work in the complex $L^2(\mathbb{S}^1)$ space. 
 The Fourier series of $df_2$ is
 \[ df_2 = \sum_{n \in \Z \setminus \{0\}} i n b_n e^{i n \theta}. \]
 This can be justified (for example) by using the fact that since $df_2$ is smooth, it converges uniformly; therefore, the Fourier series of $f_2$ can be derived from that of $df_2$ by integrating term by term.  Since $f_2$ is real,
 \[ \omega(\{a_n\},\{b_n\} ) = \frac{1}{2\pi } \int f_1 d\overline{f_2} =   \frac{1}{2\pi} \int_0^{2\pi} \left( \sum_{n \in \Z \setminus \{0\}} a_n e^{in \theta} \right) \left( \sum_{m \in \Z \setminus \{0\}} - im \overline{b_m} e^{-im \theta} \right).  \]
 Now using $\overline{b_m}=b_{-m}$ and the fact that $\{ e^{i m \theta} \}_{n \in \Z \setminus \{0 \}}$ is an orthonormal basis proves the claim.
    
\medskip \noindent {\bf{Exercise} \ref{exer:standard}}.
    By Proposition \ref{pr:three_properties_compatible} we have $\omega(v,w)=g(v,w)$.  Thus 
    \begin{align*}
     \omega(x_k,x_l) & = g(x_k,Jx_l) = g(x_k,y_l)=0 \\
     \omega(y_k,y_l) & = g(y_k,J y_l) = - g(y_k,x_l) =0 \\
     \omega(x_k,y_l) & =g(x_k,Jy_l) = -g(x_k,x_l) = - \delta_{kl} 
    \end{align*}
    where $\delta_{kl}$ is the Kronecker delta function. 

\medskip \noindent {\bf{Exercise} \ref{exer:two_models_complexify}}. 
    We first show that $\Psi$ is complex linear:
    \[  \Psi(Jv)=\frac{1}{\sqrt{2}} \left(Jv - i J^2 v \right) = \frac{1}{\sqrt{2}} i \left( v-iJv \right) = i \Psi(v).   \]
    In other words $\Psi$ is complex linear with respect to the complex structure $H_J$. 

    Clearly $\Psi(v)=0$ implies $v=0$. Dimension counting implies that it is surjective.

\medskip \noindent {\bf{Exercise} \ref{ex:Hermitian}}.
  By Exercise \ref{exer:one_two_then_J_preserves} $g(Jv,Jw)=g(v,w)$ and $\omega(Jv,Jw)=\omega(v,w)$ for all $v,w$. We compute
  \begin{align*}
   g_+(\Psi(v),\Psi(w)) & = \frac{1}{2} g_+(v-iJv,w-iJw) \\
   & = \frac{1}{2} \left( g(v,w) + g(iJv,iJw) \right) - \frac{1}{2}  \left( g(v,iJw) + g(iJv,w) \right)\\
   & = \frac{1}{2} \left( g(v,w) + g(Jv,Jw) \right) + \frac{1}{2}  \left( i g(v,Jw) -i g(Jv,w) \right)\\
   & = g(v,w) - i g(Jv,w) = g(v,w) - i \omega(v,w) \\
   & = \langle v,w \rangle.
  \end{align*}  

\medskip \noindent {\bf{Exercise}   \ref{exer:standard_polarization_Hb}}.
  The expression for $\omega$ is just the (implicit) complex linear extension of (\ref{eq:omega_in_Hb}).  Similarly, the expression for $g$ is just the sesquilinear extension of (\ref{eq:g_in_Hb}). We have 
  \[ J \{ b_m \} =    \{ \hat{b}_m \} \]
  where
  \[ \hat{b}_m = \left\{ \begin{array}{rl} -i \, \overline{b_{-m}} & m >0 \\ i \, \overline{b_{-m}} & m < 0. \end{array} \right. \]
  So we compute 
  \begin{align*}
   \omega(\{a_n\},J \alpha \{b_n\}) & = -i \sum_{n=-\infty}^\infty n a_n \hat{b}_{-n} \\
   & = -i \sum_{n=1}^\infty n a_n (i \overline{b_n}) -i \sum_{n=-\infty}^{-1} -i n a_n \overline{b_n} \\
   & = g(\{a_n\},\{b_n \}).
  \end{align*}
  The fact that $L^+ \oplus L^-$ is a direct sum decomposition follows from the definition of $\bose$.
  Assume that $a_n = 0$ for all $n<0$ and $b_n = 0$ for all $n>0$.  Then 
  \[  g(\{a_n\},\{b_n \}) = \sum_{n=-\infty}^\infty n a_n \overline{b_n} = 0. \]
  On the other hand, if $a_n=0,b_n=0$ for all $n<0$ and $b_n = 0$ for all $n<0$, 
  \[ \omega(\{a_n\},\{b_n\}) = \sum_{n=-\infty}^\infty n a_n b_{-n} =0 \]
  and similarly if $a_n=0,b_n=0$ for all $n>0$.  

\medskip \noindent {\bf{Exercise} \ref{exer:Kpol_other_half_pos_def}}.  Let $v,w \in W^-$, and observe that 
\[ i\omega(v,\alpha{w})= \overline{-i\omega(\alpha{v},w)}. \]
Since $\alpha{v},\alpha{w} \in W^+$, the claim follows from the fact that $g$ is positive-definite on $W^+$. 

\medskip \noindent {\bf{Exercise} \ref{exer:unitaries_are_same}}.
    Assume that $u \in U(H)$. Then for all $v,w \in H$ 
    \[  g_J(uv,uw) = g(uv,uw) - i\omega(uv,uw)=g(v,w)-i\omega(v,w) = g_J(u,v).  \]
    Conversely, if $g_J(uv,uw) = g_J(v,w)$
    for all $v,w \in H$, then $u$ must preserve the real and imaginary parts of $g_J$.  So $u \in \O(H) \cap \Sp(H)$.  

\medskip \noindent {\bf{Exercise} \ref{exer:stabilizer}}.
    An operator $u \in \Sp(H)$ preserves $L^{+}$ if and only if it commutes with $J$, which it does if and only if $u \in \O(H)$; whence the stabilizer of $L^{+} \in \Pol^{\omega}(H)$ is $\U(H)$.
    Let $W^{+} \in \Pol^{\omega}(H)$ be arbitrary.
    Pick an orthonormal basis $\{ e_{i} \}$ for $W^{+}$ and observe that $\{ \alpha e_{i} \}$ is an orthonormal basis for $\alpha(W^{+})$.
    Let $u$ be the complex-linear extension of the map that sends $l_{i}$ to $e_{i}$ and $\alpha l_{i}$ to $\alpha e_{i}$.
    It is straightforward to see that $u \in \Sp(H)$, and that $u(L^{+}) = W^{+}$.
    
\medskip \noindent {\bf{Exercise}
\ref{exer:VerifyGraphGivesSubspace}}.
Let $x \in W^{+}$ be arbitrary.
Set $y = P^{+}x$.
We then have $x = (P^{+}x, P^{-}x) = (y, P^{-} (P^{+}|_{W^{+}})^{-1}y)$, thus $x \in \mathrm{graph}(Z)$.
On the other hand, let $y \in L^{+}$ be arbitrary, and set $x = (P^{+}|_{W^{+}})^{-1}$.
We then have $(y,Zy) = (P^{+}x,ZP^{+}x) = (P^{+}x,P^{-}x) = x$.

\medskip \noindent {\bf{Exercise} \ref{exer:holomorphic_one_forms}}.  
The one-forms $\eta_k$ were assumed to be harmonic, so $\beta_k$ are all harmonic. By definition of $Z$, for each $k=1,\ldots,\mathfrak{g}$ there is a holomorphic one-form whose periods agree with those of $\beta_k$. By uniqueness of the harmonic representative $\beta_k$ is holomorphic. 

It remains to be shown that $\beta_k$ are linearly independent. But if 
\[ \sum_{k=1}^{\mathfrak{g}} \lambda_k \beta_k =0, \]
then integrating over each of the curves $a_k$, using the definition of $\eta_k$ we obtain that $\lambda_k=0$ for $k=1,\ldots,\mathfrak{g}$. 

\medskip \noindent {\bf{Exercise} \ref{exer:sympletic_inverse}}. Since $u$ is invertible, we need only show that the expression in equation (\ref{eq:symp_inverse_formula}) is a left inverse for $u$. This follows immediately from equation (\ref{eq:symplectomorphism_identities}). 

\medskip \noindent {\bf{Exercise} \ref{exer:complete_Hilbert_Schmidt_char}}. 
  We compute using Proposition \ref{le:top_left_invertible_symplecto} that
  \[   u^{-1} J u - J = i \left( \begin{array}{cc} a^* a + b^* b -\1 & a^* \alpha b \alpha + b^* \alpha a \alpha  \\ - \alpha b^* \alpha a - \alpha a^* \alpha b & - \alpha ( b^* b + a^* a -\1 ) \alpha \end{array} \right).  \]

  By equation (\ref{eq:symplectomorphism_identities}) we have 
  \[ a^{*}\alpha b \alpha = b^{*} \alpha a \alpha \ \ \text{and} \ \ a^*a - \1 = b^*b.  \]
  Inserting these in the above proves the claim.

\medskip \noindent {\bf{Exercise} \ref{ex:EquivalenceRelation}}.
    We have that $W_{1}^{+} \rightarrow \alpha(W_{1}^{+})$ is the zero map, thus $\sim$ is reflexive.
    Suppose that $W_{1}^{+} \sim W_{2}^{+}$.
    Denote by $P_{i}^{\pm}$ the orthogonal projection $H \rightarrow W_{i}^{\pm}$.
    We have that $\iota_{i}^{\pm} = (P_{i}^{\pm})^{*}: W_{i}^{\pm} \rightarrow H$ is the inclusion.
    This means that the orthogonal projection $W_{1}^{+} \rightarrow \alpha(W_{2}^{+})$ factors as $P_{2}^{-} \iota_{1}^{+}$; by assumption, this operator is Hilbert-Schmidt, thus so is its adjoint $(P_{2}^{-} \iota_{1}^{+})^{*} = P_{1}^{+} \iota_{2}^{-}$.
    Conjugating this operator by $\alpha$, we obtain another Hilbert-Schmidt operator
    \begin{equation*}
        \alpha P_{1}^{+} \iota_{2}^{-} \alpha = \alpha P_{1}^{+} \alpha^{2} \iota_{2}^{-} \alpha = P_{1}^{-} \iota_{2}^{+},
    \end{equation*}
    but the expression on the right-hand side is nothing but the orthogonal projection $W_{2}^{+} \rightarrow \alpha(W_{1}^{+})$, so $W_{2}^{+} \sim W_{1}^{+}$, i.e.~the relation is symmetric.
    Now, suppose that $W_{1}^{+} \sim W_{2}^{+}$ and $W_{2}^{+} \sim W_{3}^{+}$.
    This implies that the operators $P_{2}^{-}\iota_{1}^{+}$ and $P_{3}^{-}\iota_{2}^{+}$ are Hilbert-Schmidt.
    We have
    \begin{equation*}
        P_{3}^{-}\iota_{1}^{+} = P_{3}^{-}(\iota_{2}^{+}P_{2}^{+} + \iota_{2}^{-}P_{2}^{-})\iota_{1}^{+} = P_{3}^{-} \iota_{2}^{+} P_{2}^{+} \iota_{1}^{+} + P_{3}^{-} \iota_{2}^{-} P_{2}^{-} \iota_{1}^{+}.
    \end{equation*}
    It follows that $P_{3}^{-} \iota_{1}^{+}$ is Hilbert-Schmidt, and we are done. 

\medskip \noindent  {\bf{Exercise} \ref{ex:PolarCondition}}.
Indeed, if $(x,y)$ is perpendicular to both $\gra(Z)$ and $\alpha(\gra(Z))$ we have
\begin{align*}
    0&= \langle (x,y), (v, Zv) \rangle + \langle (x,y),(\alpha Zw,\alpha w) \rangle = \langle (x,y), (v, Zv) \rangle + \langle (x,y),(-Z^{*} \alpha w,\alpha w) \rangle \\
    &= \langle x, v - Z^{*} \alpha w \rangle + \langle y, Zv + \alpha w \rangle
\end{align*}
for all $v, w \in L^{+}$.
By setting $v=x$ and $\alpha w = -Zx$ we obtain
\begin{equation*}
    0 = \langle x, x + Z^{*} Z x \rangle = \| x \|^{2} + \| Z x \|^{2}.
\end{equation*}
which implies that $x=0$.
Similarly, one shows that $y=0$.

\small
\bibliographystyle{alphaurlmod}
\bibliography{bibfile}

\end{document}